\documentclass[12pt,leqno]{amsart}

\usepackage[top=35mm, bottom=35mm, left=30mm, right=30mm]{geometry}
\usepackage{mathptmx}
\usepackage{mathrsfs}
\usepackage{amsmath,amssymb,amscd,verbatim,color}
\usepackage{verbatim}
\usepackage{url}
\usepackage[all]{xy}
\usepackage{xcolor}
\usepackage[colorlinks=true,citecolor=blue]{hyperref}
\usepackage{cite}
\usepackage[latin9]{inputenc}
\usepackage{enumerate}
\usepackage{graphicx}
\usepackage{placeins}
\usepackage{float}
\usepackage{subfigure}
\usepackage{epstopdf}

\newtheorem*{TH1}{Theorem 1}
\newtheorem*{TH2}{Theorem 2}

\newtheorem{Theorem}{Theorem}[section]
\newtheorem{Lemma}{Lemma}[section]

\newtheorem{Corollary}{Corollary}[section]

\newtheorem{Problem}{Problem}[section]
\newtheorem{Claim}{Claim}[section]

\theoremstyle{definition}
\newtheorem{Remark}[Theorem]{Remark}
\newtheorem{Conjecture}{Conjecture}[section]

\newcommand{\mdim}[0]{\operatorname{mdim}}
\DeclareMathOperator{\widim}{widim}
\DeclareMathOperator{\diam}{diam}
\DeclareMathOperator{\supp}{supp}

\DeclareMathOperator{\ord}{ord}

\begin{document}

\title[The embedding problem in topological dynamics and Takens' theorem]
{The embedding problem in topological dynamics \\and Takens' theorem}

\author[Y. Gutman]{Yonatan Gutman}
\address{Y. Gutman: Institute of Mathematics, Polish Academy of Sciences,
ul. \'Sniadeckich 8, 00-656 Warszawa, Poland}
\email{y.gutman@impan.pl}

\author[Y. Qiao]{Yixiao Qiao}
\address{Y. Qiao: Institute of Mathematics, Polish Academy of Sciences,
ul. \'Sniadeckich 8, 00-656 Warszawa, Poland
 -- \and -- 
Department of Mathematics, University of Science and Technology of China,
Hefei, Anhui 230026, P.R. China}
\email{yxqiao@impan.pl}

\author[G. Szab\'o]{G\'abor Szab\'o}
\address{G. Szab\'o: Department of Mathematical Sciences, Copenhagen University, Universitetsparken 5, 2100 Copenhagen \O, Denmark}
\email{gabor.szabo@math.ku.dk}



\date{\today}

\begin{abstract}
We prove that every $\mathbb{Z}^{k}$-action $(X,\mathbb{Z}^{k},T)$ of mean dimension less than $D/2$ admitting a factor $(Y,\mathbb{Z}^{k},S)$ of Rokhlin dimension not greater than $L$ embeds in
$(([0,1]^{(L+1)D})^{\mathbb{Z}^{k}}\times Y,\sigma\times S)$, where $D\in\mathbb{N}$, $L\in\mathbb{N}\cup\{0\}$ and $\sigma$ is the shift on the Hilbert cube $([0,1]^{(L+1)D})^{\mathbb{Z}^{k}}$; in particular, when $(Y,\mathbb{Z}^{k},S)$ is an irrational $\mathbb{Z}^{k}$-rotation on the $k$-torus, $(X,\mathbb{Z}^{k},T)$ embeds in
$(([0,1]^{2^kD+1})^{\mathbb{Z}^k},\sigma)$, which is compared to a previous result by the first named author, Lindenstrauss and Tsukamoto. Moreover, we give a complete and detailed proof of Takens' embedding theorem with a continuous observable for $\mathbb{Z}$-actions and deduce the analogous result for $\mathbb{Z}^{k}$-actions. Lastly, we show that the Lindenstrauss--Tsukamoto conjecture for $\mathbb{Z}$-actions holds generically, discuss
an analogous conjecture for $\mathbb{Z}^{k}$-actions in \cite{GutQiaTsu2017} by the first two authors and Tsukamoto and verify it for $\mathbb{Z}^{k}$-actions on finite dimensional spaces.
\end{abstract}
\maketitle

\section{Introduction}
One of the fundamental problems in the field of dynamical systems is that of finding good universal spaces or models. Given a family of dynamical systems $\mathcal{C}$ we would like to find a simple as possible system which exhibits all members of $\mathcal{C}$ as subsystems. This system is referred to as $\mathcal{C}$-universal.

In this article, we will be concerned with topological dynamical systems. The simplest example is given by a pair $(X,T)$,
where $X$ is a compact metric space and $T:X\to X$ is a homeomorphism, i.e., a continuous bijective mapping from $X$ to itself. However one may consider more general group actions $(X,G,\Phi)$, where $G$ is a topological group with identity element $e$, $X$ is a compact metric space and $\Phi: G\times X \to X$ is a continuous mapping satisfying that $\Phi(e,x)=x$ and $\Phi(h,\Phi(g,x))=\Phi(hg,x)$ for any $x\in X$ and $g,h\in G$. Usually, we abbreviate $(X,G,\Phi)$ and $\Phi(g,x)$ to $(X,G)$ and $gx$ respectively. Note that in the current article we consider continuous actions. This is stronger than assuming that the action is measurable but weaker than assuming that it is smooth. The case of $(X,T)$ thus
corresponds to a $\mathbb{Z}$-action $(X,\mathbb{Z})$ and other especially interesting cases involve $G=\mathbb{R}$ or $G=\mathbb{Z}^{k}$ with $k\geq2$.

The universal systems we will consider are the $d$-cubical shifts $\mathcal{S}_{d}$ on the Hilbert cube $([0,1]^{d})^{\mathbb{Z}^{k}}$, where $d$ is a positive integer. The phase space is $([0,1]^{d})^{\mathbb{Z}^{k}}$ and the action $\sigma$ is given by shifting $(x_{n})_{n\in\mathbb{Z}^{k}}\in ([0,1]^{d})^{\mathbb{Z}^{k}}$ to $(x_{n+m})_{n\in\mathbb{Z}^{k}}\in ([0,1]^{d})^{\mathbb{Z}^{k}}$ for every $m\in\mathbb{Z}^{k}$. Arguably these systems are concrete and simple. A dynamical system $(X,\mathbb{Z}^{k})$ is isomorphic to a subsystem of $\mathcal{S}_{d}$ if and only if there exists an embedding from $(X,\mathbb{Z}^{k})$ into $\mathcal{S}_{d}$. By an \textbf{embedding} from $(X,\mathbb{Z}^{k})$ into $\mathcal{S}_{d}$, denoted by $(X,\mathbb{Z}^{k})\hookrightarrow\mathcal{S}_{d}$, we mean a continuous injective mapping $f:X\to([0,1]^{d})^{\mathbb{Z}^{k}}$ with $f\circ\Phi(n,x)=\sigma(n,f(x))$ for all $n\in\mathbb{Z}^{k}$ and $x\in X$. We are thus led to the following fundamental problem.

\begin{Problem}\label{problem}
Let $k,d$ be positive integers and $(X,\mathbb{Z}^{k})$ a dynamical system.
Does it embed into the $d$-cubical shift $\mathcal{S}_{d}$?
\end{Problem}

This problem has a long and fascinating history which we will detail below. However let us first relate this problem to Takens' theorem. The celebrated Takens theorem gives sufficient conditions under which the dynamics of a system can be reconstructed from time series of observable quantities. In many cases it lets one reconstruct the internal dynamics of a complicated nonlinear system from a single time series. The framework of Takens' theorem may be described in the following way: Given a system $X$ and an evolution rule $T$, one seeks an observable $h:X\to[0,1]$ so that the mapping $X\to[0,1]^{\ell+1},x\mapsto(h(x),h(Tx),\dots,h(T^{\ell}x))$ is one-to-one for some $\ell\geq1$.
This enables the experimentalist possessing time series $h(x_{0}),h(Tx_{0}),h(T^{2}x_{0}),\dots$ (for some points $x_{0}\in X$) to plot the following points
$$\left(h(x_{0}),h(Tx_{0}),\dots,h(T^{\ell}x_{0})\right),$$
$$\left(h(Tx_{0}),h(T^{2}x_{0}),\dots,h(T^{\ell+1}x_{0}\right),$$
$$\vdots$$
$$\left(h(T^{M-\ell}x_{0}),h(T^{M-\ell+1}x_{0}),\dots,h(T^{M}x_{0})\right),$$
$$\vdots$$
in $[0,1]^{\ell+1}$ and thus to obtain an approximation of the system as well as its dynamics. Takens \cite[Theorem 1]{T81} proved a mathematical theorem which made this approximation procedure credible. In his setting, the phase space $X$ was assumed to be a manifold, and the rule $T$ and the observable $h$ were assumed to be $C^{2}$ maps. It enables
experimentalists to construct models for complex and non-linear systems using a single observable. The applicability to non-linear systems is paramount as many other techniques in the literature are of limited use. It is thus no surprise that Takens' theorem has been used widely in experimental sciences, in particular, in physics and biology\cite{kostelich1990noise,hsieh2005distinguishing, sugihara1990nonlinear}.

Let us now relate Takens'  theorem and Problem \ref{problem}. If we assume that the observable $h:X\to[0,1]$ is continuous, then a system for which Takens' theorem holds may be equivariantly embedded into the $1$-cubical shift $\mathcal{S}_{1}$ via the mapping $X\to[0,1]^{\mathbb{Z}},x\mapsto(h(T^{i}x))_{i\in\mathbb{Z}}$ \footnote{Indeed, if $X\to[0,1]^{\ell+1},x
\mapsto(h(x),h(Tx),\dots,h(T^{\ell}x))$ is already one-to-one, a fortiori $X\to[0,1]^{\mathbb{Z}},x\mapsto(h(T^{i}x))_{i\in\mathbb{Z}}$ is one-to-one.}. The first named author \cite{Gut16} generalized Takens' theorem to the setting of a
$\mathbb{Z}$-action $(X,T)$ and a continuous observable $h$ below, showing that for a generic continuous function
$h:X\to[0,1]$ the mapping $X\to[0,1]^{2d+1},x\mapsto (h(x),h(Tx),\dots,h(T^{2d}x))$ is an embedding, where $X$ has Lebesgue covering dimension $d$ (see \eqref{lebesgue dimension} in Section 2 for the definition).

\begin{TH1}[Cf. {\cite[Theorem 1.1]{Gut16}}]
Let $d\in\mathbb{N}\cup\{0\}$ and $m\in\mathbb{N}$. Let $X$ be a compact metric space and $T:X\rightarrow X$ a homeomorphism. Assume that $\dim(X)=d$ and $\dim(P_{n})<mn/2$ for all $1\leq n\leq2d$, where $P_{n}$ denotes the set of periodic points of period $\leq n$. Then it is a generic property that the following map
\begin{equation}\label{eq:delay observation map-11}
h_{0}^{2d}:X\to([0,1]^{m})^{2d+1},\;\;x\mapsto\left(h(x),h(Tx),\dots,h(T^{2d}x)\right)
\end{equation}
is an embedding, i.e., the set of functions in $C(X,[0,1]^{m})$ for which \eqref{eq:delay observation map-11} is an embedding is comeagre w.r.t. supremum topology.
\end{TH1}

This also generalized in certain aspects versions of Takens' theorem proven by Sauer, Yorke and Casdagli \cite{SYC91}, and Robinson \cite{robinson2001rigorous,Rob05,Rob11}, where $X$ was assumed to be a compact subspace of Euclidean space, respectively Hilbert space,
with bounded box dimension. Indeed, there are spaces with bounded Lebesgue covering dimension and infinite box dimension. However the proof given in \cite{Gut16} was not complete. Our first goal is to provide a complete and detailed proof for Theorem 1.

\medskip

Let us now review the history of Problem \ref{problem}.  According to a classical theorem due to Bebutov and Kakutani \cite{K68} (see also \cite[Chapter 13]{A}), a real flow whose fixed point set is homeomorphic to a subset of $\mathbb{R}$ embeds into the space of all continuous functions on $\mathbb{R}$, with the natural action of $\mathbb{R}$. For an explicit compact universal space for all compact real flows we refer the reader to \cite{GutJin16}. Auslander \cite[p.193]{A} asked in the early 70's whether Problem \ref{problem} has a solution in the case $k=d=1$ for minimal systems \footnote{A system $(X, T)$ is said to be minimal if the orbit $\{T^{n}x: n\in \mathbb{Z}\}$ is dense in $X$ for every $x\in X$.}. It is obvious that if the set of periodic points of period $n$ of $(X,T)$ cannot be embedded into $[0,1]^{n}$ for some $n$, then $(X,T)$ cannot be embedded into $([0,1]^{\mathbb{Z}},\sigma)$. This is a reason why Auslander restricted Problem \ref{problem} to the setting of minimal systems. In 1974, Jaworski \cite{J74} answered Problem \ref{problem} positively for finite dimensional aperiodic
\footnote{A system $(X,T)$ is called aperiodic if $T^{n}x\neq x$ for all $x\in X$ and nonzero integer $n$.} systems in the case $k=d=1$. In 1991, Nerurkar \cite{nerurkar1991} generalized Jaworski's result to the case that $X$ is finite dimensional and $T$ does not have infinitely many periodic points with same period. In 2000, Lindenstrauss and Weiss \cite{LW} solved Auslander's question in the negative by using the theory of mean dimension. Mean dimension is an invariant of topological dynamical systems introduced by Gromov \cite{G} in 1999. Heuristically, \textit{it counts the number of real-valued parameters per unit time, just like topological entropy counts the number of bits per unit time needed for describing a system.} The mean dimension of $(X,\mathbb{Z}^{k})$ is denoted by $\mdim(X)$, see Section 3 for the exact definition. Not surprisingly, if the topological entropy of a system is finite, then its mean dimension is zero \cite[Section 4]{LW}. The usefulness of this invariant presents itself by considering the mean dimension of the $d$-cubical shift $\mathcal{S}_{d}$. The $d$-cubical shift is obviously infinite dimensional and of infinite topological entropy; however, its mean dimension is $d$. As both finite dimensionality and finite topological entropy imply zero mean dimension, we see that mean dimension provides nontrivial information for ``large'' systems. Lindenstrauss and Weiss developed the fundamental theory of mean dimension and applied it to several problems in topological dynamics, such as the embedding problem (as we mentioned before) and characterization of small boundary property \cite{LW,L99}.
As every system embedding into $\mathcal{S}_{d}$ has mean dimension no more than $d$, mean dimension becomes another obstruction of embedding into $d$-cubical shifts.  A construction of infinite minimal dynamical system whose mean dimension is strictly greater than $1$ was given in \cite[Proposition 3.5]{LW}; it follows that this system cannot be embedded into $(([0,1])^{\mathbb{Z}},\sigma)$, i.e., Problem \ref{problem} with $k=d=1$ has a negative answer for such a system.

In a sequel to \cite{LW}, Lindenstrauss \cite[Theorem 5.1]{L99} gave a partial converse to the necessary condition $\mdim(X)\leq d$: If $(X,T)$ is an extension of an aperiodic minimal system with $\mdim(X)<m/36$, then one can embed $(X,T)$ into $(([0,1]^{m})^\mathbb{Z},\sigma)$. In particular, for any minimal system of mean dimension strictly less than $m/36$, Problem \ref{problem} has a positive answer in the case $k=1$ and $d=m$. Another nice question related to this marvellous result, which was posed by Lindenstrauss in \cite[p.229]{L99}, is to decide the largest constant $c$ such that $\mdim(X)<cm$ implies that $(X,T)$ embeds into $(([0,1]^{m})^\mathbb{Z},\sigma)$. Recently, the first named author and Tsukamoto \cite[Theorem 1.4]{GutTsu16} proved that every minimal system $(X,T)$ of mean dimension strictly less than $m/2$ embeds into the $m$-cubical shift $\mathcal{S}_{m}$. Previously, Lindenstrauss and Tsukamoto \cite[Theorem 1.3]{Lt12} constructed a minimal system of mean dimension $m/2$ which cannot be embedded into the $m$-cubical shift $\mathcal{S}_{m}$. Combining these two results together, we get that the constant $c=1/2$ is optimal. For general $\mathbb{Z}$-actions with periodic points, Lindenstrauss and Tsukamoto conjectured that
\begin{Conjecture}[{\cite[Conjecture 1.2]{Lt12}}]\label{conjectureLinTsu}
Let $(X,T)$ be a dynamical system. If
\[
\mdim(X)<\frac{m}{2},\;\;\;\;
\frac{\dim\left(\{x:T^{n}x=x\}\right)}{n}<\frac{m}{2}\;\;\text{for all}\;\;n\ge1,
\]
then there is an embedding from $(X,T)$ into $(([0,1]^{m})^{\mathbb{Z}},\sigma)$.
\end{Conjecture}
This conjecture holds generically, see Appendix A. But it is still open in general. Note however that Theorem 1 implies that this conjecture holds for finite dimensional systems. In addition, Amyot \cite[Proposition 26]{amy2014} gave sufficient conditions for embedding of countable products of finite dimensional systems into cubical shifts .

\medskip

It was pointed out in \cite[Introduction]{GutLinTsu15}: ``The original motivation of Gromov was to apply the ideas of mean dimension to infinite dimensional dynamical systems in geometric analysis. In most situations in geometry the acting groups are more complicated than $\mathbb{Z}$. For example, when one studies a dynamical system consisting of holomorphic curves $f:\mathbb{C}\to\mathbb{C}P^{N}$ (see \cite{T2011}), the acting groups are $\mathbb{C}$ and its lattice $\mathbb{Z}^{2}$. In \cite[Chapter 4]{G} Gromov discussed a system of complex subvarieties in $\mathbb{C}^{n}$. In that case, $\mathbb{C}^{n}$ and its lattice $\mathbb{Z}^{2n}$ are the acting groups, the action being by translation. So geometry naturally requires us to extend the theory of mean dimension from $\mathbb{Z}$ to more general groups, specifically $\mathbb{Z}^{k}$." It therefore makes sense to study the relation between mean dimension and the embedding problem for $\mathbb{Z}^{k}$-actions. Nonetheless, Lindenstrauss already remarked that the obstruction to extending his results in \cite{L99} for $\mathbb{Z}$-actions to the setting of $\mathbb{Z}^{k}$-actions is not ``purely technical''. Indeed, the embedding problem for $\mathbb{Z}^{k}$-actions $(k\geq 2)$ has proven itself to be more difficult than the $\mathbb{Z}$ case. There are essentially two known results, both appearing in \cite{GutLinTsu15}:
With a relatively easy proof modelled on \cite{GutTsu12} it was shown that $(X,\mathbb{{Z}}^{k})$ of $\mdim(X)<m/2$ admitting an aperiodic symbolic factor embeds into the $m$-cubical shift $\mathcal{S}_{m}$ \cite[Theorem 1.6]{GutLinTsu15}; and with a hard and very technical proof it was proven that if $(X,\mathbb{{Z}}^{k})$ is an extension
of an infinite minimal system and satisfies $\mdim(X)<{m}/{2^{k+1}}$, then there exists an embedding from $(X,\mathbb{{Z}}^{k})$ into $(([0,1]^{2m})^{\mathbb{Z}^{k}},\sigma)$ \cite[Theorem 1.5]{GutLinTsu15}.
It turns out that the most substantial progress for the embedding problem for $\mathbb{{Z}}^{k}$-actions involves very hard proofs and it treats only free \footnote{A $\mathbb{Z}^{k}$-action $(X,\mathbb{Z}^{k})$ is called free, if for all $x\in X$, $nx=x$ implies $n=0$.} systems.
In order to tackle the even harder case involving quasi-periodic \footnote{A point $x$ of $\mathbb{Z}^{k}$-action $(X,\mathbb{Z}^{k})$ is called quasi-periodic, if there exists $n\neq 0$ such that $nx=x$.} points it seems advisable to start by finding simpler proofs in the aperiodic case.
A promising direction involves Rokhlin dimension for $\mathbb{{Z}}^{k}$-actions, a notion which arose in the context of the classification of transformation group $C^*$-algebras \cite{TomsWinter13, HWZ12, Sza15, WinterClassifying, SzaboWuZacharias}, i.e., $C^*$-algebras associated to topological dynamical systems via the crossed product construction.
The topological definition is due to Winter and appears first in \cite[Definition 2.1]{Sza15} relying on\cite{HWZ12}: A $\mathbb{Z}^{k}$-action $(X,\mathbb{Z}^{k})$ is said to have (topological) \textbf{Rokhlin dimension} $d$,
written $\dim_{Rok}(X,\mathbb{Z}^{k})=d$, if $d$ is the smallest nonnegative integer such that for each $n\in\mathbb{{N}}$, one can find $d+1$ open sets $U_{0},\dots,U_{d}$ satisfying the following properties:

\begin{enumerate}
\item for each $0\leq i\leq d$, $U_{i}$ induces an $[n]$-tower, where $[n]=\{0,1,\dots,n-1\}^{k}\subset\mathbb{Z}^{k}$;
namely, $\{g\overline{U}_{i}\}_{g\in[n]}$ are pairwise disjoint (see Figure \ref{Rokhlin tower} below);

\item the union of the $d+1$ $[n]$-towers covers the whole space $X$: $\bigcup_{i=0}^{d}\bigcup_{g\in[n]}gU_{i}=X$.
\end{enumerate}
If there is no such $d$, then we say that $(X,\mathbb{Z}^{k})$ has infinite Rokhlin dimension and
write $\dim_{Rok}(X,\mathbb{Z}^{k})=\infty$.
\begin{figure}[H]
    \centering
    \includegraphics[width=4.0in]{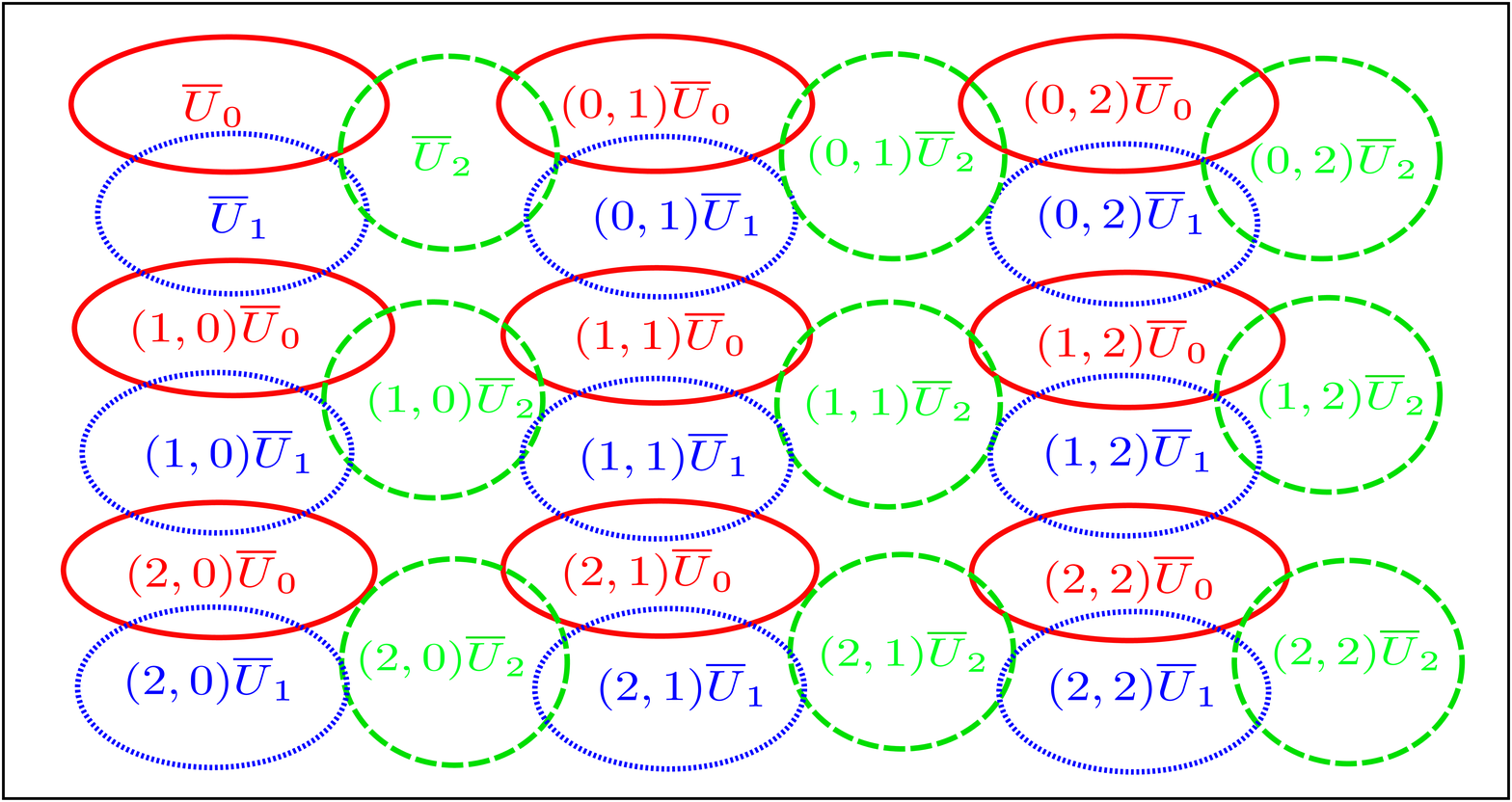}
    \caption{[3]-towers for a $\mathbb{Z}^{2}$-action of Rokhlin dimension 2.}
   \label{Rokhlin tower}
\end{figure}
\begin{Remark}
If a system is not free, then its Rokhlin dimension is $\infty$.
\end{Remark}

This definition is attractive as one can try to solve the embedding problem tower by tower similarly to what has been done in \cite{GutTsu12} and \cite[Section 1.6]{GutLinTsu15}. In addition, since the towers are allowed to overlap, the definition allows for connected systems, unlike the case of \cite{GutTsu12} and \cite[Section 1.6]{GutLinTsu15} where the system must have a zero dimensional factor which implies strong unconnectedness. In this article, we obtain the following result on the embedding problem for $\mathbb{Z}^{k}$-actions with a simple and conceptually appealing proof:

\begin{TH2}\label{embedding general}
Let $D\in\mathbb{N}\cup\{0\}$ and $L\in\mathbb{N}$. Let $(X,\mathbb{Z}^{k},T)$ be an extension of $(Y,\mathbb{Z}^{k},S)$ with the factor map $\pi:X\to Y$. Assume that $\dim_{Rok}(Y,\mathbb{Z}^{k})=D$ and $\mdim(X)<L/2$. Then the set of functions
\[
\left\{f\in C(X,[0,1]^{(D+1)L}): I_f\times\pi~\text{is an embedding}\right\}
\]
is a dense $G_\delta$ subset of $C(X,[0,1]^{(D+1)L})$. In particular, there exists an embedding from $(X,\mathbb{Z}^{k},T)$ into $\bigl(([0,1]^{(D+1)L})^{\mathbb{Z}^k}\times Y,\sigma\times S\bigl)$.
\end{TH2}

As we mentioned previously, in \cite[Theorem 1.5]{GutLinTsu15} the authors proved that if $(X,\mathbb{{Z}}^{k})$ is an extension of an infinite minimal system and satisfies $\mdim(X)<{D}/{2^{k+1}}$ then there exists an embedding from $(X,\mathbb{{Z}}^{k})$ into $(([0,1]^{2D})^{\mathbb{Z}^{k}},\sigma)$. Comparing this result to Corollary \ref{irrr} below,
we see that our result improves upon this by a factor of $2$ for systems admitting minimal irrational rotations as factors.

\begin{Corollary}\label{irrr}
Let $(X,\mathbb{Z}^{k})$ be an extension of an irrational $\mathbb{Z}^{k}$-rotation on the $k$-torus
\footnote{An irrational $\mathbb{Z}^{k}$-rotation on the $k$-torus $\mathbb{T}^{k}=(\mathbb{R}/\mathbb{Z})^k$ is given by
$(n_{1},\dots,n_{k})\times(x_{1},\dots, x_{k})\mapsto(x_{1}+n_{1}\alpha_{1},\dots, x_{k}+n_{k}\alpha_{k})$,
where $(n_{1},\dots,n_{k})\in\mathbb{Z}^{k}$, $(x_{1},\dots, x_{k})\in\mathbb{T}^{k}$, $\alpha_{1},\dots,\alpha_{k}$ are irrational numbers but are not necessarily linearly independent over the rationals.}
with $\mdim(X)<L/2$. Then there exists an embedding from
$(X,\mathbb{Z}^{k})$ into $\bigl(([0,1]^{2^kL+1})^{\mathbb{Z}^{k}},\sigma\bigl)$.
\end{Corollary}

This paper is organized as follows. In Section 2, we provide a complete and detailed proof for Takens' embedding theorem with a continuous observable for $\mathbb{Z}$-actions (Theorem 1). In Section 3, we establish a new condition implying embeddability into $d$-cubical shifts $\mathcal{S}_{d}$ for systems admitting factors of bounded Rokhlin dimension (Theorem 2). In Section 4, we state an analogy of Conjecture \ref{conjectureLinTsu} for $\mathbb{Z}^{k}$-actions and
verify its correctness for finite dimensional $\mathbb{Z}^{k}$-actions by generalizing Theorem 1 to the setting of $\mathbb{Z}^{k}$-actions. In Appendix A, we prove that Conjecture \ref{conjectureLinTsu} holds generically.

\vspace{0.4cm}

\textbf{Acknowledgements.}
Y. Gutman was partially supported by the Marie Curie grant PCIG12-GA-2012-334564 and
the National Science Center (Poland) grants 2013/08/A/ST1/00275 and 2016/22/E/ST1/00448.
Y. Qiao was partially supported by NSF of China (grant numbers 11371339 and 11571335).
G. Szab{\'o} was supported by SFB 878 \emph{Groups, Geometry and Actions}, EPSRC grant EP/N00874X/1
and the Danish National Research Foundation through the Centre for Symmetry and Deformation (DNRF92).
A visit of G. Szab{\'o} to IMPAN was partially financed by WCMCS (Poland). The authors would like to thank the referee for his/her careful reading and helpful suggestions.

\section{Takens' embedding theorem with a continuous observable \\
for $\mathbb{Z}$-actions}
We begin with necessary notions and basic results. For a compact metric space $X$, we denote by $\mathcal{C}$ the collection of all finite open covers of $X$. Given $\alpha\in\mathcal{C}$ and $x\in X$, we can count the number of elements in $\alpha$ to which $x$ belongs, i.e.,
$$\#\{U:x\in U\in\alpha\}=\sum_{U\in\alpha}1_{U}(x).$$
The order of $\alpha$, denoted by $\ord(\alpha)$, is essentially defined by maximizing this quantity:
$$\ord(\alpha)=-1+\max_{x\in X}\sum_{U\in \alpha}1_{U}(x).$$
Let $D(\alpha)=\min_{\beta\succ\alpha}\ord(\beta)$, where $\beta\succ\alpha$ means that for every $V\in\beta$ there is $U\in\alpha$ with $V\subset U$. The \textbf{Lebesgue covering dimension} is defined as

\begin{equation*}
\dim(X)=\sup_{\alpha\in\mathcal{C}}D(\alpha).
\end{equation*}
Define the \textbf{mesh} of a finite open cover $\alpha$ of $X$ by $\text{mesh}(\alpha)=\max_{U\in\alpha}\diam(U)$. It is not hard to show that for any $\epsilon>0$,
\begin{equation}\label{lebesgue dimension}
\dim(X)=\sup_{\alpha\in\mathcal{C},\text{mesh}(\alpha)<\epsilon}D(\alpha).
\end{equation}

\medskip

The main tool of the proofs in this paper is the Baire category theorem. A \textbf{Baire space} is a topological space where the intersection of countably many dense open sets is dense.  Note that $(C(X,[0,1]^{m}),\|\cdot\|_{\infty})$ is a complete metric space. By the Baire category theorem \cite[Theorem 8.4]{kechris2012classical}, $(C(X,[0,1]^{m}),\|\cdot\|_{\infty})$ is a Baire space. A set in a topological space is said to be \textbf{comeagre} or \textbf{generic} if it is the complement of a countable union of nowhere dense sets. A set is called $\mathbf{G_{\delta}}$ if it is a countable intersection of open sets. Note that a dense $G_{\delta}$ set is comeagre.

Let $(X,T)$ be a topological dynamical system. For every $n\geq 1$, define
$$P_{n}=\big\{x\in X: \; T^{i}x=x \;\text{ for some }\; 1\leq i\leq n \big\}$$
the set of all periodic points of period less than or equal to $n$, and
$$H_{n}=P_{n}\setminus P_{n-1}\;\;\;\;(P_{0}=\emptyset)$$
the set of all periodic points of period $n$, and
$$P=\cup_{n\geq1} H_{n}$$
the set of all periodic points. For $h\in C(X,[0,1]^{m})$ and $d\in\mathbb{N}\cup\{0\}$, we define
$$h_{0}^{2d}:X\to([0,1]^{m})^{2d+1},\;\;x\mapsto\left(h(x),h(Tx),\dots,h(T^{2d}x)\right).$$
Let $K$ be a compact subset of $(X\times X)\setminus\triangle$, where $\triangle=\{(x,x):x\in X\}$ is the diagonal of $X\times X$. Set
$$D_{K}=\left\{h\in C(X,[0,1]^{m}):h_{0}^{2d}(x)\neq h_{0}^{2d}(y), (x,y)\in K\right\}.$$

\begin{Theorem}[=Theorem 1]\label{takens embedding for Z}
Let $d\in\mathbb{N}\cup\{0\}$ and $m\in\mathbb{N}$.
Let $X$ be a compact metric space and $T:X\rightarrow X$ a homeomorphism.
Assume that $\dim(X)=d$ and $\dim(P_{n})<mn/2$
for all $1\leq n\leq2d$.
Then it is a generic property that the following map
\begin{equation}\label{eq:delay observation map-1}
h_{0}^{2d}:X\to([0,1]^{m})^{2d+1},\;\;x\mapsto\left(h(x),h(Tx),\dots,h(T^{2d}x)\right)
\end{equation}
is an embedding, i.e., the set of functions in $C(X,[0,1]^{m})$
for which \eqref{eq:delay observation map-1} is an embedding is comeagre w.r.t. supremum topology.
\end{Theorem}

\begin{proof}[Outline of the proof of Theorem \ref{takens embedding for Z}:]
By the Baire category theorem, it suffices to show that $(X\times X)\setminus\triangle$ can be covered by countably many compact subsets $K_{1},K_{2},\dots$ such that each $D_{K_{i}}$ is open and dense in $C(X,[0,1]^{m})$. To achieve this, we may try to find open neighbourhoods $U_{(x,y)}$ in $X\times X$ of $(x,y)$ for every pair $(x,y)\in (X\times X)\setminus\triangle$ such that $D_{\overline{U}_{(x,y)}}$ is open and dense in $C(X,[0,1]^{m})$ and then extract a countable subcover of $\{U_{(x,y)}: (x,y)\}$ of $(X\times X)\setminus\triangle$ \footnote{Throughout this paper, all the unspecified closures are taken in $X$.}. It is not hard to show that $D_{K}$ is open in $C(X,[0,1]^{m})$ for each compact subset
$K$ of $(X\times X)\setminus\triangle$ (see Part 1 in the proof below).
Unfortunately, we are not able to find such open neighbourhoods for all pairs $(x,y)\in(X\times X)\setminus\triangle$,
e.g., $(x,y)\in P\times P$. Instead we divide the whole space $(X\times X)\setminus\triangle$ to the following cases:

\textbf{Case (A). The orbits of $x$ and $y$ are disjoint.}
\begin{itemize}
  \item Case (A.1). $x,y\in P_{2d}$ and their orbits are disjoint;
  \item Case (A.2). $x,y\in \cup_{n\geq 2d+1}P_{n}$ and their orbits are disjoint;
  \item Case (A.3). $x,y\in X\setminus P$ and their orbits are disjoint;
  \item Case (A.4). $x\in P_{2d}$, $y\in\cup_{n\geq 2d+1}P_{n}$ or vice verse, and therefore their orbits are disjoint;
  \item Case (A.5). $x\in P_{2d}$, $y\in X\setminus P$ or vice verse, and therefore their orbits are disjoint;
  \item Case (A.6). $x\in \cup_{n\geq 2d+1}P_{n}$, $y\in X\setminus P$ or vice verse, and therefore their orbits are disjoint.
\end{itemize}

\textbf{Case (B). Both $x$ and $y$ are periodic and their orbits intersect.}
\begin{itemize}
  \item Case (B.1). $x,y\in P_{2d}$ and their orbits intersect;
  \item Case (B.2). $x,y\in \cup_{n\geq 2d+1}P_{n}$ and their orbits intersect.
\end{itemize}

\textbf{Case (C). Both $x$ and $y$ are aperiodic and their orbits intersect.}

For each of these cases, we find a set $U_{(x,y)}$ containing $(x,y)$ such that $D_{\overline{U}_{(x,y)}}$ is dense in
$C(X,[0,1]^{m})$; moreover, the sets $U_{(x,y)}$ ($(x,y)\in (X\times X)\setminus\triangle$) are open w.r.t. the following subspaces equipped with the subspace topology:

\begin{itemize}
\item in Cases (A.1), (A.2), (A.4), (B.1) and (B.2), $U_{(x,y)}$ is open in $H_{m_{1}}\times H_{m_{2}}$ for some $m_{1},m_{2}\in\mathbb{N}$;
\item in Cases (A.3) and (C), $U_{(x,y)}$ is open in $\big((X\setminus P)\times(X\setminus P)\big)\setminus\triangle$;
\item in Cases (A.5) and (A.6), $U_{(x,y)}$ is open in $H_{m_{3}}\times(X\setminus P)$ or $(X\setminus P)\times H_{m_{4}}$ for some $m_{3},m_{4}\in\mathbb{N}$.
\end{itemize}
Note that every subspace of $(X\times X)\setminus\triangle$ is a Lindel\"{o}f space \footnote{A Lindel\"{o}f space is a topological space where every open cover has a countable subcover.}. Using the Lindel\"{o}f property of the following subspaces
$$(H_{m_{1}}\times H_{m_{2}})\setminus\triangle, \;\; H_{m_{3}}\times(X\setminus P),\;\;
(X\setminus P)\times H_{m_{4}},\;\;
\big((X\setminus P)\times(X\setminus P)\big)\setminus\triangle,$$
whose union covers $(X\times X)\setminus\triangle$, where $m_{1},m_{2},m_{3},m_{4}\in\mathbb{N}$,
we may find a countable closed cover $\mathcal{U}$ of $(X\times X)\setminus\triangle$ such that for each $K\in\mathcal{U}$, $D_{K}$ is open and dense in $C(X,[0,1]^{m})$.
\end{proof}

In the proof below,
for every positive integer $N$,
the coordinates of a vector $v\in ([0,1]^{m})^{N}$ is numbered from $0$ to $N-1$, i.e., $v=(v_{0},v_{1},\dots,v_{N-1})$,
where $v_{i}\in[0,1]^{m}$. We define $v|_{s}=v_{s}$ and $v|^{s}_{r}=(v_{i})_{i=r}^{s}$ for any $0\leq r\leq s\leq N-1$. For a finite subset $\mathcal{F}=\{x_{i}\}_{i\in I}$ of $([0,1]^{m})^{N}$, the convex hull of $\mathcal{F}$ is defined as
$$\text{co}(\mathcal{F})=\left\{\sum_{i\in I}\lambda_{i}x_{i}: \sum_{i\in I}\lambda_{i}=1, \lambda_{i}\geq0\right\}.$$

\begin{proof}[Proof of Theorem \ref{takens embedding for Z}]
\textbf{Part 1.} We prove that for every compact set $K\subset (X\times X)\setminus\triangle$, $D_{K}$ is open in $C(X,[0,1]^{m})$.

Fix $K$ and $h\in D_{K}$. Define $\alpha:K\to[0,1]$ by
$$\alpha(x,y)=\max_{0\leq i\leq2d}||h(T^{i}x)-h(T^{i}y)||_{\infty}$$
for every $(x,y)\in K$. Obviously, $\alpha$ is continuous on $K$. Since $K$ is a compact subset of $(X\times X)\setminus\triangle$, $\alpha$ attains its minimum on $K$. Assume that its minimum on $K$ is $a$. Since $\alpha(x,y)>0$ for each $(x,y)\in K$, we get $a>0$. For any $g\in C(X,[0,1]^{m})$ with $||g-h||_{\infty}<a/3$ and any $(x,y)\in K$, there exists $0\leq i_{0}\leq 2d$ such that $||h(T^{i_{0}}x)-h(T^{i_{0}}y)||_{\infty}\geq a$ and hence
\begin{align*}
&||g(T^{i_{0}}x)-g(T^{i_{0}}y)||_{\infty}\\
\geq & ||h(T^{i_{0}}x)-h(T^{i_{0}}y)||_{\infty}-||g(T^{i_{0}}x)-h(T^{i_{0}}x)||_{\infty}-||g(T^{i_{0}}y)-h(T^{i_{0}}y)||_{\infty}\\
> & ||h(T^{i_{0}}x)-h(T^{i_{0}}y)||_{\infty}-\frac{2a}{3}\geq \frac{a}{3}.
\end{align*}
This implies $g\in D_{K}$.

\textbf{Part 2.}
We denote the period of every $x\in X$ by $p_{x}\in\mathbb{N}\cup\{+\infty\}$ (if $x$ is aperiodic, we write $p_{x}=+\infty$) and define the \textbf{adjusted period} of $x$ by
$$\tilde{p}_{x}=\min\{2d+1,p_{x}\}.$$
We now consider the cases (A), (B) and (C). Fix $(x,y)\in (X\times X)\setminus\triangle$.

\textbf{Case (A). The orbits of $x$ and $y$ are disjoint.}
In particular, $x,Tx,\dots,T^{\tilde{p}_{x}-1}x$, $y,Ty,\dots,T^{\tilde{p}_{y}-1}y$ are pairwise distinct.

Case (A.1). $x,y\in P_{2d}$. By the definition, $x\in H_{p_{x}}$ ($p_{x}<2d+1$) and $y\in H_{p_{y}}$ ($p_{y}<2d+1$). We can choose $\epsilon>0$ such that
$$\overline{B}(x,\epsilon),T\overline{B}(x,\epsilon),\dots,T^{p_{x}-1}\overline{B}(x,\epsilon),
\overline{B}(y,\epsilon),T\overline{B}(y,\epsilon),\dots,T^{p_{y}-1}\overline{B}(y,\epsilon)$$
are pairwise disjoint. Since $P_{p_{x}-1}$ is closed in $X$ and $x\notin P_{p_{x}-1}$, $d(x,P_{p_{x}-1})>0$.
(If $p_{x}=1$, set $d(x, \emptyset)=+\infty$.) Observe that $H_{p_{x}}$ is open in $P_{p_{x}}$ and $x\in H_{p_{x}}$. One may select $0<\eta_{1}<\min\{\epsilon,d(x,P_{p_{x}-1})\}$ so that $B(x,\eta_{1})\cap P_{p_{x}}\subset H_{p_{x}}$. Similarly, there exists $0<\eta_{2}<\min\{\epsilon,d(y,P_{p_{y}-1})\}$ such that $B(y,\eta_{2})\cap P_{p_{y}}\subset H_{p_{y}}$. Set
$$U_{x}=B(x,\eta_{1})\cap P_{p_{x}},\;\;\;\;\; U_{y}=B(y,\eta_{2})\cap P_{p_{y}}.$$
Obviously, $U_{x}$ (resp. $U_{y}$) is open in $H_{p_{x}}$ (resp. $H_{p_{y}}$).
One can check that
$$\overline{U}_{x}\subset\overline{B}(x,\eta_{1})\cap P_{p_{x}}=\overline{B}(x,\eta_{1})\cap H_{p_{x}}\subset H_{p_{x}}$$
and
$$\overline{U}_{y}\subset\overline{B}(y,\eta_{2})\cap P_{p_{y}}=\overline{B}(y,\eta_{2})\cap H_{p_{y}}\subset H_{p_{y}}.$$

Case (A.2). $x,y\in \cup_{n\geq 2d+1}P_{n}$. Similarly to Case (A.1), we can take open neighbourhoods
$U_{x}$ and $U_{y}$ of $x$ and $y$ in $H_{p_{x}}$ and $H_{p_{y}}$ respectively such that
$\overline{U}_{x},T\overline{U}_{x},\dots,T^{2d}\overline{U}_{x}$, $\overline{U}_{y},T\overline{U}_{y},\dots,T^{2d}\overline{U}_{y}$ are pairwise disjoint, and $\overline{U}_{x}\subset H_{p_{x}}$ and $\overline{U}_{y}\subset H_{p_{y}}$.

Case (A.3). $x,y\in X\setminus P$. We can choose $\epsilon>0$ such that
$$\overline{B}(x,\epsilon),T\overline{B}(x,\epsilon),\dots,T^{2d}\overline{B}(x,\epsilon),
\overline{B}(y,\epsilon),T\overline{B}(y,\epsilon),\dots,T^{2d}\overline{B}(y,\epsilon)$$
are pairwise disjoint. Set $U_{x}=B(x,\epsilon)\cap(X\setminus P)$ and $U_{y}=B(y,\epsilon)\cap(X\setminus P)$.

Case (A.4). $x\in\cup_{n\geq 2d+1}P_{n}$ and $y\in P_{2d}$. Similarly to Case (A.1), we can take open neighbourhoods $U_{x}$ and $U_{y}$ of $x$ and $y$ in $H_{p_{x}}$ and $H_{p_{y}}$ respectively such that
$\overline{U}_{x},T\overline{U}_{x},\dots,T^{2d}\overline{U}_{x}$, $\overline{U}_{y},T\overline{U}_{y},\dots,T^{p_{y}-1}\overline{U}_{y}$ are pairwise disjoint, and $\overline{U}_{x}\subset H_{p_{x}}$ and $\overline{U}_{y}\subset H_{p_{y}}$.

Case (A.5). $x\in X\setminus P$ and $y\in P_{2d}$. Similarly to Cases (A.1) and (A.3), we can take open neighbourhoods $U_{x}$ and $U_{y}$ of $x$ and $y$ in $X\setminus P$ and $H_{p_{y}}$ respectively such that
$$\overline{U}_{x},T\overline{U}_{x},\dots,T^{2d}\overline{U}_{x},\overline{U}_{y},T\overline{U}_{y},\dots,T^{p_{y}-1}\overline{U}_{y}$$
are pairwise disjoint and $\overline{U}_{y}\subset H_{p_{y}}$.

Case (A.6). $x\in X\setminus P$ and $y\in \cup_{n\geq 2d+1}P_{n}$. Similarly to Case (A.5), there exists open neighbourhoods $U_{x}$ and $U_{y}$ of $x$ and $y$ in $X\setminus P$ and $H_{p_{y}}$ respectively such that
$\overline{U}_{x},T\overline{U}_{x},\dots,T^{2d}\overline{U}_{x}$, $\overline{U}_{y},T\overline{U}_{y},\dots,T^{2d}\overline{U}_{y}$
are pairwise disjoint and $\overline{U}_{y}\subset H_{p_{y}}$.

Set $K_{(x,y)}=\overline{U}_{x}\times\overline{U}_{y}$. In the following we show that $D_{K_{(x,y)}}$ is dense in $C(X,[0,1]^{m})$. Let $\epsilon>0$ and $\tilde{f}\in C(X,[0,1]^{m})$. We will show that there exists $f\in C(X,[0,1]^{m})$
so that $\|f-\tilde{f}\|_{\infty}<\epsilon$ and $f\in D_{K_{(x,y)}}$.

The facts $\dim(X)=d$ and $\dim(P_{n})<mn/2$ for all $1\leq n\leq2d$ imply $\dim(\overline{U}_{j})<\tilde{p}_{j}m/2$ for $j\in\{x,y\}$. By \eqref{lebesgue dimension}, one can choose finite open covers $\alpha_{x}$ and $\alpha_{y}$ of $\overline{U}_{x}$ and $\overline{U}_{y}$ respectively such that for $j\in\{x,y\}$ it holds that
$$\max_{W\in\alpha_{j},0\leq k\leq \tilde{p}_{j}-1}\diam(\tilde{f}(T^{k}W))<\frac{\epsilon}{2},\;\;\;\;
\ord(\alpha_{j})<\frac{\tilde{p}_{j}m}{2}.$$
For each $W\in\alpha_{j}$ choose $q_{W}\in W$ so that $\{q_{W}\}_{W\in\alpha_{j}}$ is a collection of distinct points in $\overline{U}_{j}$ and define $\tilde{v}_{W}=(\tilde{f}(T^{k}q_{W}))_{k=0}^{\tilde{p}_{j}-1}\in ([0,1]^{m})^{\tilde{p}_{j}}$ (see Figure \ref{disjointcase}). Without loss of generality, we assume $\tilde{p}_{x}\geq \tilde{p}_{y}$.

\begin{figure}[H]
    \centering
    \includegraphics[width=3.0in]{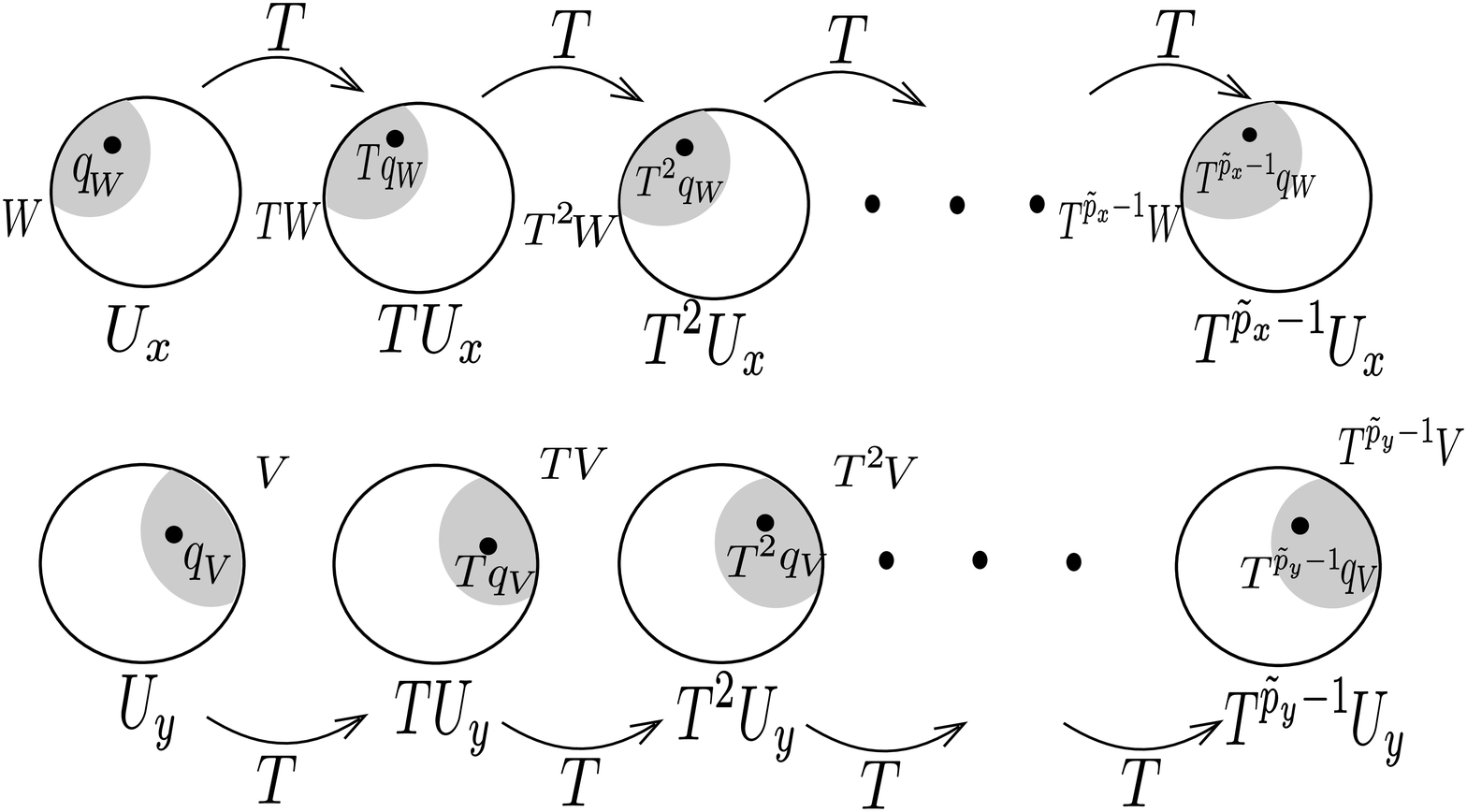}
     \caption{Case (A)}
   \label{disjointcase}
\end{figure}
\begin{Claim}\label{claim1}
For $j\in\{x,y\}$, there exists a continuous function $F_{j}:\overline{U}_{j}\rightarrow ([0,1]^{m})^{\tilde{p}_{j}}$ such that  the following properties hold:
\begin{enumerate}[(1)]
\item for any $W\in\alpha_{j}$, $\|F_j(q_{W})-\tilde{v}_{W}\|_{\infty}<\epsilon/2$;

\item for any $z\in\overline{U}_{j}$,
$F_{j}(z)\in co\{F_{j}(q_{W}):z\in W\in\alpha_{j}\}$;

\item if $x'\in\overline{U}_{x}$ and $y'\in\overline{U}_{y}$,
then $F_{x}(x')\neq F_{y}(y')^{\oplus \tilde{p}_{x}}$, where $F_{y}(y')^{\oplus\tilde{p}_{x}}:
\overline{U}_{y}\rightarrow ([0,1]^{m})^{\tilde{p}_{x}}$
is the function given by the formula $F_{y}(y')^{\oplus\tilde{p}_{x}}|_k=F_{y}(y')|_{k\text{ mod }\tilde{p}_{y}}$
for $0\leq k\leq \tilde{p}_{x}-1$.
\end{enumerate}
\end{Claim}
Let us assume Claim \ref{claim1} and complete the proof as follows. Set $A=\bigcup_{j\in \{x,y\}}\bigcup_{k=0}^{\tilde{p}_{j}-1}T^{k}\overline{U}_{j}$. Using the fact that $\overline{U}_{x},T\overline{U}_{x},\dots,T^{\tilde{p}_{x}-1}\overline{U}_{x}$,
$\overline{U}_{y},T\overline{U}_{y},\dots,T^{\tilde{p}_{y}-1}\overline{U}_{y}$ are pairwise disjoint, we define $f':A\rightarrow [0,1]^{m}$ by $$f'_{|T^{k}\overline{U}_{j}}(T^{k}z)=F_{j}(z)|_{k}$$ for every $j\in\{x,y\}$, $0 \leq k\leq\tilde{p}_{j}-1$ and $z\in\overline{U}_{j}$.

Now we show $\|f'-\tilde{f}_{|A}\|_{\infty}<\epsilon$. Fix $j\in\{x, y\}$. Take $z\in\overline{U}_{j}$ and $0 \leq k\leq\tilde{p}_{j}-1$. By property (2), we have
$$f'(T^{k}z)=F_{j}(z)|_{k}\in co\{F_{j}(q_{W})|_{k}:z\in W\in\alpha_{j}\},$$
and hence
$$||f'(T^{k}z)-\tilde{f}(T^{k}z)||_{\infty}\leq\max_{z\in W\in\alpha_{j}}||F_{j}(q_{W})|_{k}-\tilde{f}(T^{k}z)||_{\infty}.$$
Thus, to prove $\|f'-\tilde{f}_{|A}\|_{\infty}<\epsilon$ it suffices to show $||F_{j}(q_{W})|_{k}-\tilde{f}(T^{k}z)||_{\infty}<\epsilon$.
Fix $z\in W\in\alpha_{j}$. Note that
$$||F_{j}(q_{W})|_{k}-\tilde{f}(T^{k}z)||_{\infty}\leq||F_{j}(q_{W})|_{k}-\tilde{v}_{W}|_{k}||_{\infty}+||\tilde{v}_{W}|_{k}-\tilde{f}(T^{k}z)||_{\infty}.$$
By property (1), $||F_{j}(q_{W})|_{k}-\tilde{v}_{W}|_{k}||_{\infty}\leq\epsilon/2$. Because $\diam(\tilde{f}(T^{k}W))<\epsilon/2$ for every $W\in\alpha_{j}$ and $0 \leq k\leq\tilde{p}_{j}-1$, we have
$$||\tilde{v}_{W}|_{k}-\tilde{f}(T^{k}z)||_{\infty}=||\tilde{f}(T^{k}q_{W})-\tilde{f}(T^{k}z)||_{\infty}<\epsilon/2.$$
Therefore $$||F_{j}(q_{W})|_{k}-\tilde{f}(T^{k}z)||_{\infty}<\epsilon.$$
By Lemma \ref{continuous extension} below, there is a continuous function $f:X\rightarrow [0,1]^{m}$ so that $f{}_{|A}=f'$ and $\|f-\tilde{f}\|_{\infty}<\epsilon$.

Now we prove $f\in D_{K_{(x,y)}}$. Assume that for some $(x',y')\in K_{(x,y)}$, we have $f_{0}^{2d}(x')=f_{0}^{2d}(y')$. In Cases (A.2), (A.3) and (A.6), $\tilde{p}_{x}=\tilde{p}_{y}=2d+1$, and therefore by the definition of $F_{y}(y')^{\oplus\tilde{p}_{x}}$ and $f$, we know
$$F_{y}(y')^{\oplus\tilde{p}_{x}}=F_{y}(y')=f_{0}^{2d}(y')=f_{0}^{2d}(x')=F_{x}(x'),$$
a contradiction to property (3).

In Cases (A.1), (A.4) and (A.5), note that $\tilde{p}_{y}=p_{y}$ and $y'\in\overline{U}_{y}\subset H_{p_{y}}$. It follows from the definition of $F_{y}(y')^{\oplus\tilde{p}_{x}}$ that for every $0\leq k\leq \tilde{p}_{x}-1$,
$$F_{y}(y')^{\oplus\tilde{p}_{x}}|_{k}=F_{y}(y')|_{(k \text{ mod } p_{y})}=f(T^{(k \text{ mod } p_{y})}y')=f(T^{k}y'),$$
where in the last equality we use $T^{p_{y}}y'=y'$,
and hence $$F_{y}(y')^{\oplus\tilde{p}_{x}}=(f(y'),f(Ty'),\dots,f(T^{\tilde{p}_{x}-1}y')).$$
Obviously, $$F_{x}(x')=(f(x'),f(Tx'),\dots,f(T^{\tilde{p}_{x}-1}x')).$$
So the assumption $f_{0}^{2d}(x')=f_{0}^{2d}(y')$ implies
$$F_{x}(x')=F_{y}(y')^{\oplus\tilde{p}_{x}},$$ a contradiction to property (3). This ends the proof of Case (A).

\medskip

The remaining task for Case (A) is to verify Claim \ref{claim1}. In fact, for $j\in\{x,y\}$ let $\{\psi_{W}\}_{W\in\alpha_{j}}$ be a partition of unity subordinate to $\alpha_{j}$; that is, a collection of continuous functions $\psi_{W}:\overline{U}_{j}\to [0,1]$ such that
$$\sum_{W\in \alpha_j}\psi_{W}(z)=1\;\;\text{for all}\;\;z\in\overline{U}_{j}$$
and $\supp(\psi_{W})\subset W$, and we can further assume that $\psi_{W}(q_{W})=1$ for all $W\in\alpha_{j}$.  Set
$\vec{v}_{W}=\tilde{v}_{W}$ for all $W\in\alpha_{y}$. Let $\vec{v}_{W}\in([0,1]^{m})^{\tilde{p}_{x}}$ be vectors that will be specified later and will be approximately equal to $\tilde{v}_{W}$ for all $W\in\alpha_{x}$. We define $F_{j}:\overline{U}_{j}\to ([0,1]^{m})^{\tilde{p}_{j}}$ for $j\in\{x,y\}$ as follows:
$$F_{j}(z)=\sum_{W\in\alpha_{j}}\psi_{W}(z)\vec{v}_{W}.$$
Note that for every $W\in\alpha_{j}$,
$$F_{j}(q_{W})=\vec{v}_{W}.$$
For any $z\in\overline{U}_{j}$, define $\alpha_{j,z}=\{W\in\alpha_{j}:\psi_{W}(z)>0\}$. Property (3) is equivalent to the following inequality:

\begin{equation}\label{explicit equality-1-1 for two periodic points}
\sum_{W\in\alpha_{x,x'}}\psi_{W}(x')\vec{v}_{W}\neq\sum_{W\in\alpha_{y,y'}}\psi_{W}(y')(\vec{v}_{W})^{\oplus\tilde{p}_{x}}.\end{equation}
Note that the total number of vectors in \eqref{explicit equality-1-1 for two periodic points} is bounded from above by
the number of elements in $\alpha_{x,x'}\cup\alpha_{y,y'}$, which is not more than $\tilde{p}_{x}m+1$.
Set
$$V_{x'}=\text{span}\left\{\vec{v}_{W}:W\in\alpha_{x,x'}\right\},\;\;\;V_{y'}^{\oplus\tilde{p}_{x}}=\text{span}\left\{(\vec{v}_{V})^{\oplus\tilde{p}_{x}}:V\in\alpha_{y,y'}\right\}.$$
By the definition of $\alpha_{j,z}$, we have $\dim(V_{x'})\leq\ord(\alpha_{x})+1\leq(\tilde{p}_{x}m+1)/2$ and
$\dim(V_{y'}^{\oplus\tilde{p}_{x}})\leq\ord(\alpha_{y})+1\leq(\tilde{p}_{y}m+1)/2$.

Fix $x'$ and $y'$. Set $r=\dim(V_{y'}^{\oplus\tilde{p}_{x}})$, $s=\dim(V_{x'})$ and $m'=\tilde{p}_{x}m$.
If $r=(\tilde{p}_{y}m+1)/2$, then $(\vec{v}_{W})^{\oplus\tilde{p}_{x}}$ ($W\in\alpha_{y,y'}$) are linearly independent and
$r+s\leq m'+1$. By Lemma \ref{affinely independent and linear independent} (2) below we know that for almost every choice of $\vec{v}_{W}$ ($W\in\alpha_{x,x'}$) the following vectors
$$\vec{v}_{W}\;\;\text{and}\;\;(\vec{v}_{V})^{\oplus\tilde{p}_{x}}
\;\;(\text{for all}\;\;W\in\alpha_{x,x'}\;\;\text{and}\;\;V\in\alpha_{y,y'})$$ are affinely independent
\footnote{The vectors $v_{1},\dots,v_{m}\in\mathbb{R}^{n}$ are called affinely independent if for
any $\lambda_{i}\in\mathbb{R}$, $i=1,\dots,m$, the conditions $\sum_{i=1}^{m}\lambda_{i}v_{i}=0$
and $\sum_{i=1}^{m}\lambda_{i}=0$ imply $\lambda_{i}=0$ for all $i=1,\dots,m$.}.
Observe that $\sum_{W\in\alpha_{x,x'}}\psi_{W}(x')-\sum_{W\in\alpha_{y,y'}}\psi_{W}(y')=0$ and that for all
$W\in\alpha_{x,x'}$ and $V\in\alpha_{y,y'}$, $\psi_{W}(x')$ and $\psi_{V}(y')$ are positive. Therefore
\eqref{explicit equality-1-1 for two periodic points} holds for almost every choice of $\vec{v}_{W}$ ($W\in\alpha_{x,x'}$).
If $r<(\tilde{p}_{y}m+1)/2$, then $r+s\leq m'$. By Lemma \ref{affinely independent and linear independent} (1) below,
we know that for almost every choice of $\vec{v}_{W}$ ($W\in\alpha_{x,x'}$), it holds that

\begin{equation*}
\text{span}\Big\{\vec{v}_{W},W\in\alpha_{x,x'}\Big\}\cap
\text{span}\left\{(\vec{v}_{V})^{\oplus\tilde{p}_{x}},V\in\alpha_{y,y'}\right\}=\{\vec{0}\}
\end{equation*}
and that $\vec{v}_{W}$ ($W\in\alpha_{x,x'}$) are linearly independent.
This implies that for almost every choice of $\vec{v}_{W}$ ($W\in\alpha_{x,x'}$),
\eqref{explicit equality-1-1 for two periodic points} holds. As there are only a finite number of constraints of form
\eqref{explicit equality-1-1 for two periodic points}, for almost every choice of $\vec{v}_{W}$ ($W\in\alpha_{x}$),
\eqref{explicit equality-1-1 for two periodic points} holds for all $x'\in\overline{U}_{x}$ and $y'\in\overline{U}_{y}$.
Therefore we can choose $\vec{v}_{W}$ ($W\in\alpha_{x}$) such that both properties $(1)$ and $(3)$ hold. Obviously, property (2) holds. This finishes the proof of Claim \ref{claim1}.

\medskip

\textbf{Case (B). Both $x$ and $y$ are periodic and their orbits intersect.} Assume that there exists $1 \leq l\leq p_{x}-1$ such that $y=T^{l}x$.
Similarly to Case (A.1), we can choose an open neighbourhood $U$ of $x$ in $H_{p_{x}}$ such that $\overline{U},T\overline{U},\dots,T^{p_{x}-1}\overline{U}$ are pairwise disjoint and $\overline{U}\subset H_{p_{x}}$.
Set $K_{(x,y)}=\overline{U}\times T^{l}\overline{U}$. In the following we show that $D_{K_{(x,y)}}$ is dense in $C(X,[0,1]^{m})$. Let $\epsilon>0$ and $\tilde{f}\in C(X,[0,1]^{m})$. We will prove that there exists $f\in C(X,[0,1]^{m})$
so that $\|f-\tilde{f}\|_{\infty}<\epsilon$ and $f\in D_{K_{(x,y)}}$.

By assumption, $\dim(\overline{U})<\tilde{p}_{x}m/2$.
So one can choose a finite open cover $\alpha$ of $\overline{U}$ such that
$$\max_{W\in\alpha,0\leq k\leq p_{x}-1}\diam(\tilde{f}(T^{k}W))<\frac{\epsilon}{2},\;\;\; \ord(\alpha)<\frac{\tilde{p}_{x}m}{2}.$$
For each $W\in\alpha$ choose $q_{W}\in W$ so that $\{q_{W}\}_{W\in\alpha}$  is a collection of distinct points in
$\overline{U}$ and define $\tilde{v}_{W}=(\tilde{f}(T^{j}q_{W}))_{j=0}^{p_{x}-1}\in ([0,1]^{m})^{p_{x}}$ (see Figure \ref{fig: jointperiodic} below).
\begin{figure}[H]
    \centering
    \includegraphics[width=3.0in]{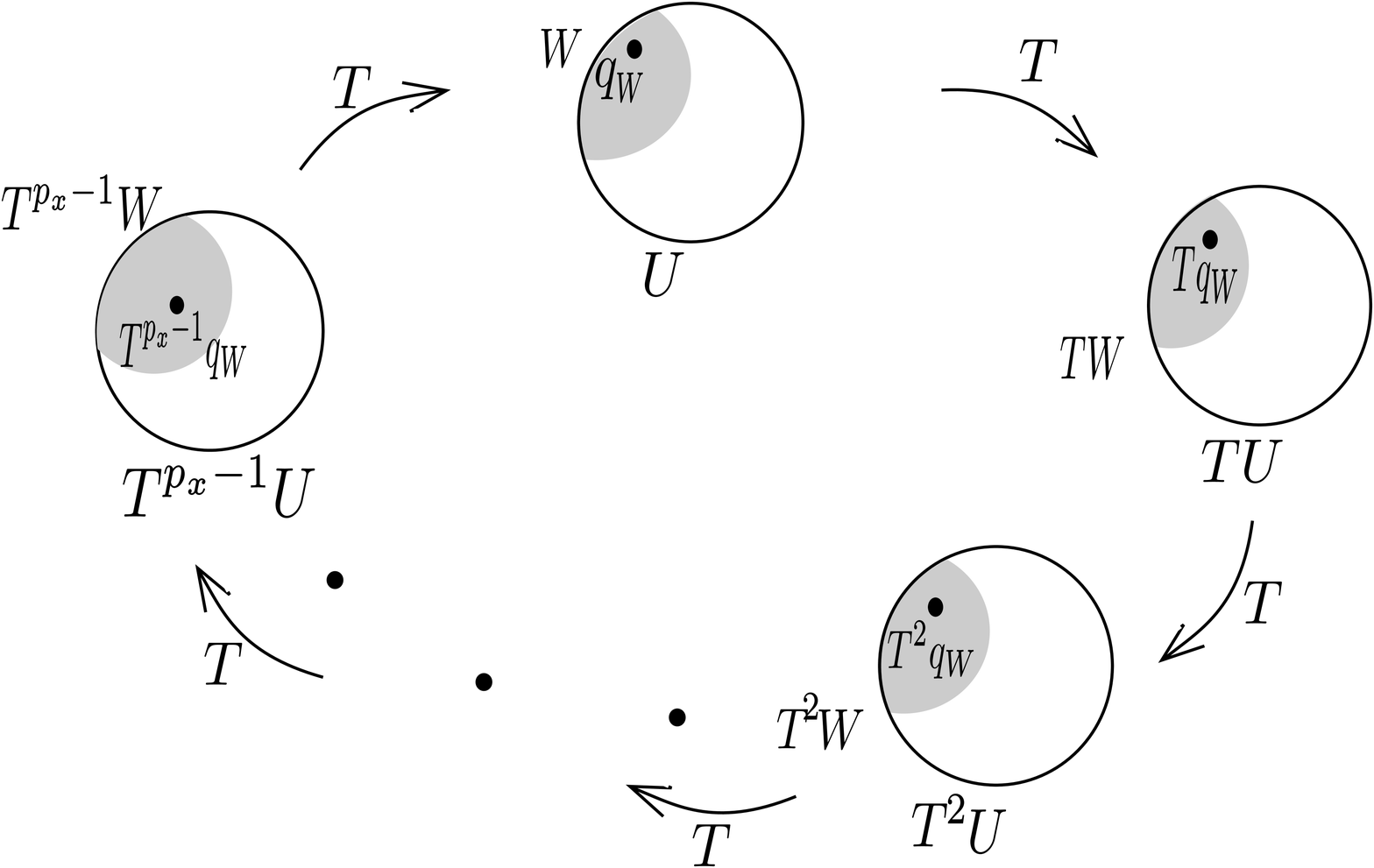}
     \caption{Case (B)}
   \label{fig: jointperiodic}
\end{figure}
\begin{Claim}\label{claim3}
There is a continuous function $F:\overline{U}\rightarrow([0,1]^{m})^{p_{x}}$ satisfying the following properties:
\begin{enumerate}[(a)]
\item for any $W\in\alpha$, $||F(q_{W})-\tilde{v}_{W}||_{\infty}<\epsilon/2$;

\item for any $z\in\overline{U}$, $F(z)\in co\{F(q_{W}):z\in W\in\alpha\}$;

\item if $x',y'\in\overline{U}$, then there exists $0\leq i\leq\tilde{p}_{x}-1$ such that $F(x')|_{i}\neq F(y')|_{(i+l)\text{ mod } p_{x}}$.
\end{enumerate}
\end{Claim}
Let us assume Claim \ref{claim3} and complete the proof as follows. Set $A=\bigcup_{i=0}^{p_{x}-1}T^{i}\overline{U}$. Using the fact that $\overline{U},T\overline{U},\dots,T^{p_{x}-1}\overline{U}$ are pairwise disjoint, we define $f':A\to [0,1]^{m}$ by
$$f'_{|T^{k}\overline{U}}(T^{k}z)=F(z)|_{k}$$
for every $0\leq k\leq p_{x}-1$ and $z\in\overline{U}$.
Similarly to Case (A), we can check $\|f'-\tilde{f}_{|A}\|_{\infty}<\epsilon$. By Lemma \ref{continuous extension} below,
we extend $f':A\to [0,1]^{m}$ to a continuous function $f:X\to [0,1]^{m}$ with $f_{|A}=f'$ and $\|f-\tilde{f}\|_{\infty}<\epsilon$.

Now we prove $f\in D_{K_{(x,y)}}$. Assume that for some $(x',y')\in K_{(x,y)}=\overline{U}\times T^{l}\overline{U}$, we have $f_{0}^{2d}(x')=f_{0}^{2d}(y')$. In particular,
$$(f(x'),f(Tx'),\dots,f(T^{\tilde{p}_{x}-1}x'))=(f(y'),f(Ty'),\dots,f(T^{\tilde{p}_{x}-1}y')).$$
By the definition of $f$, we know $$F(x')|_{0}^{\tilde{p}_{x}-1}=(f(x'),f(Tx'),\dots,f(T^{\tilde{p}_{x}-1}x')).$$
Note that $T^{-l}y'\in\overline{U}\subset H_{p_{x}}$.
For every $0\leq i\leq\tilde{p}_{x}-1$, we have:
$$f(T^{i}y')=f(T^{i+l}(T^{-l}y'))=f(T^{(i+l) \text{ mod } p_{x}}(T^{-l}y'))=F(T^{-l}y')|_{(i+l) \text{ mod } p_{x}}$$
and thus,
$$F(x')|_{i}=F(T^{-l}y')|_{(i+l) \text{ mod } p_{x}},$$
a contradiction to property (c). This ends the proof of Case (B).

\medskip

The remaining task for Case (B) is to verify Claim \ref{claim3}. In fact, let $\{\psi_{W}\}_{W\in\alpha}$ be a partition of unity subordinate
to $\alpha$; that is, a collection of continuous functions $\psi_{W}:\overline{U}\to[0,1]$ such that
$$\sum_{W\in\alpha}\psi_{W}(x)=1\;\;\text{for all}\;\;x\in\overline{U}$$
and $\supp(\psi_{W})\subset W$, and we can further assume that $\psi_{W}(q_{W})=1$ for all $W\in\alpha$.
For every $W\in\alpha$ we choose a vector $\vec{v}_{W}\in ([0,1]^{m})^{p_{x}}$ such that
$||\vec{v}_{W}-\tilde{v}_{W}||_{\infty}<\epsilon/2$.
Define $F:\overline{U}\to([0,1]^{m})^{p_{x}},\;x\mapsto\sum_{W\in\alpha}\psi_{W}(x)\vec{v}_{W}$. This function $F$ clearly satisfies properties (a) and (b). We claim that for almost every choice of $\vec{v}_{W}$, it satisfies property (c).

For every $x\in\overline{U}$, define $\alpha_{x}=\{W\in\alpha:\psi_{W}(x)>0\}$.
Fix $x',y'\in\overline{U}$.
Write (c)
explicitly as follows:
\begin{align}\label{explicit equality-1-1for periodic point}
&\sum_{U\in\alpha_{x'}}\psi_{U}(x')\left(\vec{v}_{U}|_{0},\vec{v}_{U}|_{1},\dots,\vec{v}_{U}|_{\tilde{p}_{x}-1}\right)-\\
&\sum_{V\in\alpha_{y'}}\psi_{V}(y')\left(\vec{v}_{V}|_{l\text{ mod } p_{x}},\vec{v}_{V}|_{(l+1)\text{ mod } p_{x}},
\dots,\vec{v}_{V}|_{(l+\tilde{p}_{x}-1) \text{ mod } p_{x}}\right)\neq0\nonumber.
\end{align}
Set $\alpha_{x'}=\{U_{1}^{1},\dots,U_{m_{1}}^{1}\}$ and $\alpha_{y'}=\{U_{1}^{2},\dots,U_{m_{2}}^{2}\}$ with $m_{1},m_{2}\leq\ord(\alpha)+1$. Regard the vectors $\vec{v}_{U}|_{0}^{\tilde{p}_{x}-1}, \vec{v}_{V}|_{l \text{ mod } p_{x}}^{(l+\tilde{p}_{x}-1) \text{ mod } p_{x}}\in ([0,1]^{m})^{\tilde{p}_{x}}$ as column vectors having $\tilde{p}_{x}m$ elements. Let $\mathcal{M}$ be the matrix consisting of all the vectors appearing in Equation (\ref{explicit equality-1-1for periodic point}) as follows:
$$\mathcal{M}=
\left(\begin{matrix}
{\vec{v}_{U_{1}^{1}}}|_{0}&\cdots&\vec{v}_{U_{m_{1}}^{1}}|_{0}&\vec{v}_{U_{1}^{2}}|_{l\text{ mod } p_{x}}&\cdots&\vec{v}_{U_{m_{2}}^{2}}|_{l\text{ mod } p_{x}}\\

{\vec{v}_{U_{1}^{1}}}|_{1}&\cdots&\vec{v}_{U_{m_{1}}^{1}}|_{1}&\vec{v}_{U_{1}^{2}}|_{(l+1)\text{ mod } p_{x}}&\cdots&\vec{v}_{U_{m_{2}}^{2}}|_{(l+1)\text{ mod } p_{x}}\\

\vdots&\ddots&\vdots&\vdots&\ddots&\vdots\\

\vec{v}_{U_{1}^{1}}|_{\tilde{p}_{x}-1}&\cdots&\vec{v}_{U_{m_{1}}^{1}}|_{\tilde{p}_{x}-1}&\vec{v}_{U_{1}^{2}}|_{(l+\tilde{p}_{x}-1)\text{ mod } p_{x}}&\cdots& \vec{v}_{U_{m_{2}}^{2}}|_{(l+\tilde{p}_{x}-1)\text{ mod } p_{x}}
\end{matrix}\right).$$
Note that the matrix $\mathcal{M}$ has $\tilde{p}_{x}m$ rows and no more than $\tilde{p}_{x}m+1$ columns, and that satisfies the conditions of Lemma \ref{affinely independent for matrix}.
We get that for almost every choice of $\vec{v}_{U}|_{j}$ and $\vec{v}_{V}|_{(l+j)\text{ mod } p_{x}}$ for all
$0\leq j\leq \tilde{p}_{x}-1$, $U\in\alpha_{x'}$ and $V\in\alpha_{y'}$, the matrix $\mathcal{M}$ depending on $x'$ and $y'$ consists of affinely independent columns. Though different $x'$ and $y'$ may give rise to different matrices, the total number of matrices that arise in this way is finite. Therefore for all $x,y\in\overline{U}$, for almost every choice of $\{\vec{v}_{U}\}_{U\in\alpha}$,
the matrices $\mathcal{M}$ consist of affinely independent columns, which implies
\eqref{explicit equality-1-1for periodic point}. This finishes the proof of Claim \ref{claim3}.

\medskip

\textbf{Case (C). Both $x$ and $y$ are aperiodic and their orbits intersect.}
Without loss of generality, we assume that $y=T^{l}x$ for some integer $l>0$. Since $x$ is aperiodic, $x,Tx,\dots,T^{l+2d}x$ are pairwise distinct. One may select an open neighbourhood $U$ of $x$ in $X\setminus P$ such that $\overline{U},T\overline{U},\dots,T^{l+2d}\overline{U}$ are pairwise disjoint. Set $K_{(x,y)}=\overline{U}\times T^{l}\overline{U}$.
In the following we prove that $D_{K_{(x,y)}}$ is dense in $C(X,[0,1]^{m})$. Let $\epsilon>0$ and $\tilde{f}\in C(X,[0,1]^{m})$.
We will show that there exists $f\in C(X,[0,1]^{m})$ so that $\|f-\tilde{f}\|_{\infty}<\epsilon$ and $f\in D_{K_{(x,y)}}$.

By assumption, $\dim(\overline{U})\leq\dim(X)=d<(2d+1)/{2}$. We can choose a finite open cover $\alpha$ of $\overline{U}$ such that
$$\max_{W\in\alpha,0\leq k\in\leq l+2d}\diam(\tilde{f}(T^{k}W))<\frac{\epsilon}{2},\;\;\;\; \ord(\alpha)<\frac{2d+1}{2}.$$
For each $W\in\alpha$ choose $q_{W}\in W$ so that $\{q_{W}\}_{W\in\alpha}$  is a collection of distinct points in $\overline{U}$ and define $\tilde{v}_{W}=(\tilde{f}(T^{j}q_{W}))_{j=0}^{l+2d}\in ([0,1]^{m})^{2d+l+1}$ (see Figure \ref{jointaperiodic} below).

\begin{figure}[H]
    \centering
    \includegraphics[width=4.0in]{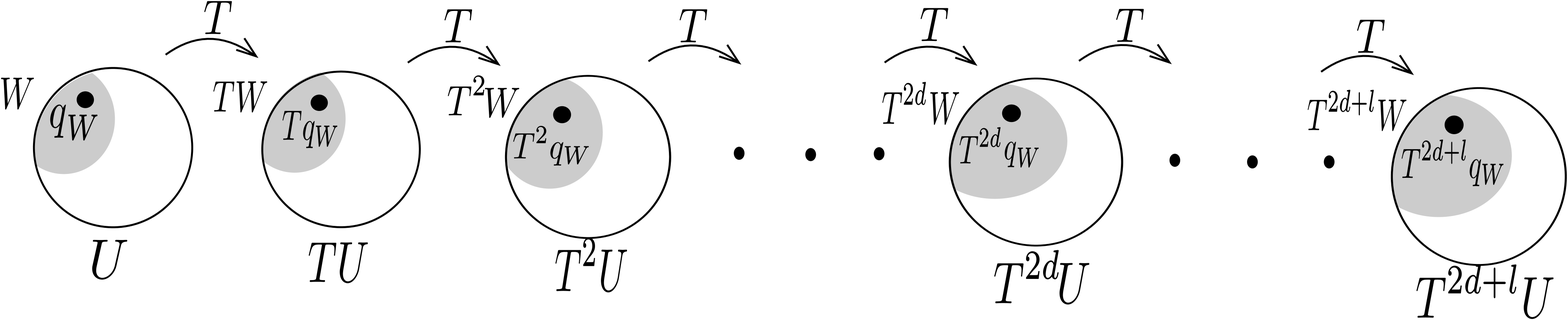}
     \caption{Case (C)}
   \label{jointaperiodic}
\end{figure}

\begin{Claim}\label{claim2}
There is a continuous function $F:\overline{U}\rightarrow ([0,1]^{m})^{l+2d+1}$ with the following properties:
\begin{enumerate}[(i)]
\item for any $W\in\alpha$, $||F(q_{W})-\tilde{v}_{W}||_{\infty}<\epsilon/2$;

\item for any $z\in\overline{U}$, $F(z)\in co\{F(q_{W}):z\in W\in\alpha\}$;

\item if $x',y'\in\overline{U}$, then $F(x')|_{0}^{2d}\neq F(y')|_{l}^{l+2d}$.
\end{enumerate}
\end{Claim}
Let us assume Claim \ref{claim2} and complete the proof as follows. Set $A=\bigcup_{i=0}^{l+2d}T^{i}\overline{U}$. Using the fact that
$\overline{U},T\overline{U},\dots,T^{l+2d}\overline{U}$ are pairwise disjoint, we define $f':A\to [0,1]^{m}$ by
$$f'_{|T^{k}\overline{U}}(T^{k}z)=F(z)|_{k}$$
for every $0\leq k\in2d+l$ and $z\in\overline{U}$. Similarly to Case (A), we can check $\|f'-\tilde{f}_{|A}\|_{\infty}<\epsilon$.
By Lemma \ref{continuous extension} below, there exists a continuous function $f:X\to [0,1]^{m}$ such that $f_{|A}=f'$ and
$\|f-\tilde{f}\|_{\infty}<\epsilon$.

Now we prove $f\in D_{K_{(x,y)}}$. Assume that for some $(x',y')\in K_{(x,y)}=\overline{U}\times T^{l}\overline{U}$,
$f_{0}^{2d}(x')=f_{0}^{2d}(y')$. By the definition of $f$, we know $$F(x')|_{0}^{2d}=(f(x'),f(Tx'),\dots,f(T^{2d}x')).$$
For every $0\leq k\leq 2d$, we get
$$f(T^{k}y')=f(T^{k+l}(T^{-l}y'))=F(T^{-l}y')|_{(k+l)}$$
and hence $F(T^{-l}y')|_{l}^{l+2d}=(f(y'),f(Ty'),\dots,f(T^{2d}y'))$. Thus,
$$F(x')|_{0}^{2d}=F(T^{-l}y')|_{l}^{l+2d},$$
a contradiction to property (iii). This ends the proof of Case (C).

\medskip

The remaining task for Case (C) is to verify Claim \ref{claim2}. In fact, let $\{\psi_{W}\}_{W\in\alpha}$ be a partition of unity subordinate
to $\alpha$; that is, a collection of continuous functions $\psi_{W}:\overline{U}\to[0,1]$ such that
$$\sum_{W\in \alpha}\psi_{W}(x)=1\;\;\text{for all}\;\;x\in\overline{U}$$
and $\supp(\psi_{W})\subset W$, and we can further assume that $\psi_{W}(q_{W})=1$ for all $W\in\alpha$.
We choose for every $W\in\alpha$ a vector $\vec{v}_{W}\in ([0,1]^{m})^{l+2d+1}$ such that
$$||\vec{v}_{W}-\tilde{v}_{W}||_{\infty}<\frac{\epsilon}{2}$$
and define $F$ by
$$F(x)=\sum_{W\in\alpha}\psi_{W}(x)\vec{v}_{W}.$$
Note that
$$F(q_{W})=\vec{v}_{W} \;\text{  for all }\; W\in\alpha.$$
This function $F$ clearly satisfies properties (i) and (ii). We claim that for almost every choice of $\vec{v}_{W}$, it satisfies property (iii).

For every $x\in\overline{U}$, define $\alpha_{x}=\{W\in\alpha:\psi_{W}(x)>0\}$. Fix $x',y'\in\overline{U}$. Write property (iii) in Claim \ref{claim2} explicitly as follows:

\begin{equation}\label{explicit equality-1-1aperiodic point}
\sum_{U\in\alpha_{x'}}\psi_{U}(x')\vec{v}_{U}|_{0}^{2d}-
\sum_{V\in\alpha_{y'}}\psi_{V}(y')\vec{v}_{V}|_{l}^{l+2d}\neq0.
\end{equation}
Set $\alpha_{x'}=\{U_{1}^{1},\dots,U_{m_{1}}^{1}\}$ and $\alpha_{y'}=\{U_{1}^{2},\dots,U_{m_{2}}^{2}\}$ with $m_{1},m_{2}\leq\ord(\alpha)+1$. Regard the vectors $\vec{v}_{U}|_{0}^{2d},\vec{v}_{V}|_{l}^{l+2d}\in ([0,1]^{m})^{2d+1}$
as column vectors containing $(2d+1)m$ elements. Let $\mathcal{M}$ be the matrix consisting of all the vectors appearing in Equation \eqref{explicit equality-1-1aperiodic point}:
\begin{equation*}
\mathcal{M}=\left(\begin{matrix}
{\vec{v}_{U_{1}^{1}}}|_{0}&\dots&\vec{v}_{U_{m_{1}}^{1}}|_{0}&\vec{v}_{U_{1}^{2}}|_{l}&\dots&\vec{v}_{U_{m_{2}}^{2}}|_{l}\\
 {\vec{v}_{U_{1}^{1}}}|_{1}&\dots&\vec{v}_{U_{m_{1}}^{1}}|_{1}&\vec{v}_{U_{1}^{2}}|_{l+1}&\dots&\vec{v}_{U_{m_{2}}^{2}}|_{l+1}\\
 \vdots&\ddots&\vdots&\vdots&\ddots&\vdots\\
 \vec{v}_{U_{1}^{1}}|_{2d}&\dots&\vec{v}_{U_{m_{1}}^{1}}|_{2d}&\vec{v}_{U_{1}^{2}}|_{l+2d}&\dots&\vec{v}_{U_{m_{2}}^{2}}|_{l+2d}
 \end{matrix}\right).
\end{equation*}
Note that the matrix $\mathcal{M}$ has $(2d+1)m$ rows and no more than $(2d+1)m+1$ columns.
Using a similar argument to Claim \ref{claim3}, we may complete the proof of Claim \ref{claim2}.

\medskip

\textbf{Part 3.}
The set $(X\times X)\setminus\triangle$ can be written as $C_{1}\cup C_{2}\cup C_{3}$,
where $$C_{1}=(P\times P)\setminus\triangle=(\cup_{m,n\in\mathbb{N}}H_{m}\times H_{n})\setminus\triangle,\;\;\;\;
C_{2}=((X\setminus P)\times (X\setminus P))\setminus\triangle$$
and
$$C_{3}=(P\times (X\setminus P))\cup((X\setminus P)\times P)
=(\cup_{m\in\mathbb{N}}H_{m}\times (X\setminus P))\cup (\cup_{m\in\mathbb{N}}(X\setminus P)\times H_{m}).$$

Fix $m,n\geq1$ and consider $H_{m}\times H_{n}$. Take $(x,y)\in (H_{m}\times H_{n})\setminus\triangle$.
By Cases (A) and (B) of Part 2, there exist open neighbourhoods $U_{x}$ of $x$ in $H_{m}$ and $U_{y}$ of $y$ in $H_{n}$
such that $D_{K_{(x,y)}}$ is dense in $C(X,[0,1]^{m})$ and $K_{(x,y)}\cap\triangle=\emptyset$, where $K_{(x,y)}=\overline{U}_{x}\times\overline{U}_{y}$. The space $X$ is second-countable, so is $X\times X$. Thus, every subspace of $(X\times X)\setminus\triangle$ is a Lindel\"{o}f space. For the open cover $\{U_{x}\times U_{y}:(x,y)\in (H_{m}\times H_{n})\setminus\triangle\}$ of $(H_{m}\times H_{n})\setminus\triangle$ we can find a countable subcover
$\{U_{x_{i}}\times U_{y_{i}}:(x_{i},y_{i})\in (H_{m}\times H_{n})\setminus\triangle,i\in\mathbb{N}\}$.
It follows that $\overline{\mathcal{U}}_{m,n}=\{\overline{U}_{x_{i}}\times\overline{U}_{y_{i}}:i\in\mathbb{N}\}$
is a countable closed cover of $(H_{m}\times H_{n})\setminus\triangle$ and thus
$\overline{\mathcal{U}}=\cup_{m,n\in\mathbb{N}}\overline{\mathcal{U}}_{m,n}$ is a countable closed cover of $C_{1}$.
Similarly, we can find countable closed covers $\overline{\mathcal{V}}$ and $\overline{\mathcal{W}}$ of $C_{2}$ and
$C_{3}$ respectively such that for all $K\in\overline{\mathcal{V}}\cup\overline{\mathcal{W}}$, $D_{K}$ is dense in $C(X,[0,1]^{m})$.

Set $\overline{\mathcal{U}}\cup\overline{\mathcal{V}}\cup\overline{\mathcal{W}}=\{K_{1},K_{2},\dots\}$. This is a countable closed cover of $(X\times X)\setminus\triangle$ and each $D_{K_{i}}$ is open in $C(X,[0,1]^{m})$ by Part 1.
By the Baire category theorem, we know that $\bigcap_{i\in\mathbb{N}}D_{K_{i}}$ is a dense $G_{\delta}$ subset of $C(X,[0,1]^{m})$. Fix $h\in\bigcap_{i\in\mathbb{N}}D_{K_{i}}$. For any pair $(x,y)\in (X\times X)\setminus\triangle$, there exists $i\in\mathbb{N}$ such that $(x,y)\in K_{i}$. It follows from the definition of $D_{K_{i}}$ that $h_{0}^{2d}(x)\neq h_{0}^{2d}(y)$.
Therefore $h_{0}^{2d}:X\to ([0,1]^{m})^{2d+1}$ is an embedding. The proof is complete.
\end{proof}




\medskip

Next we list the lemmas used in the proof of Theorem \ref{takens embedding for Z} and prove two of them.
\begin{Lemma}[{\cite[Lemmas A.6, A.8]{Gut12a}}]\label{affinely independent and linear independent}
Let $m,s\in\mathbb{N}$ and $r\in\mathbb{N}\cup\{0\}$. Suppose that $v_{1},v_{2},\dots,v_{r}\in \mathbb{R}^{m}$ are linearly independent. Then the following hold:
\begin{enumerate}
\item If $r+s\leq m$, then almost surely w.r.t. Lebesgue measure for $(v_{r+1},v_{r+2},\dots,v_{r+s})\in([0,1]^{m})^{s}$, $v_{1},v_{2},\dots,v_{r+s}$ are linearly independent.

\item If $r+s\leq m+1$, then almost surely w.r.t. Lebesgue measure for $(v_{r+1},v_{r+2},\dots,v_{r+s})\in([0,1]^{m})^{s}$,
$v_{1},v_{2},\dots,v_{r+s}$ are affinely independent.
\end{enumerate}
\end{Lemma}

\begin{Lemma}[Cf. {\cite[Lemma 5.5]{L99}}]\label{affinely independent for matrix}
Let $k,l,r$ be positive integers with $k\geq\max\{l,2\}$. Let $M=(M(i,j))_{i,j}$ be a $(k-1)\times l$ matrix with
$\{1,2,\dots,r\}=\{M(i,j):1\leq i\leq k-1,1\leq j\leq l\}$ such that no value appears twice in any row and in any column
and a value appears at most twice in $M$. Then for almost all $t_{1},t_{2},\dots,t_{r}\in[0,1]$,
the column vectors of the following matrix
$$A(t_{1},t_{2},\dots,t_{r}):=(t_{M(i,j)})_{i,j}$$
are affinely independent.
\end{Lemma}

\begin{proof}
By enlarging $r$, we may add to the matrix $k-l$ columns on the right such that the elements of the new columns are pairwise distinct and distinct from the elements appearing in the original matrix $M$. We will thus prove the lemma under the assumption that $M$ is of dimension $(k-1)\times k$. By Fubini's theorem, the statement of the lemma for the original matrix will follow.

We prove this lemma by induction on $k$. For $k=2$, without loss generality, we set $M=[1, 2]$. It is clear that for all $t_{1}\neq t_{2}\in\mathbb{R}$ (thus for almost all $t_{1},t_{2}\in[0,1]$), the column vectors $t_{1}$ and $t_{2}$ of $A(t_{1},t_{2})=[t_{1},t_{2}]$ are affinely independent. Assume that the result holds for $k\geq2$. Now we prove the case $k+1$. We have two cases in the following.

\textbf{Case 1.} No element of $M$ appears twice in $M$. Set $A(t_{1},t_{2},\dots,t_{r})=[\vec{a}_{1},\dots,\vec{a}_{k}]$, where $\vec{a}_{i}=(t_{M(j,i)})_{j=1}^{k-1}$ for $1\leq i\leq k$. By Lemma \ref{affinely independent and linear independent} (2), for almost all $\vec{a}_{1},\dots,\vec{a}_{k}\in[0,1]^{k-1}$ (hence for almost all $t_{1},\dots,t_{r}\in[0,1]$), the column vectors of $A(t_{1},\dots,t_{r})$ are affinely independent.

\textbf{Case 2.}
There exist $1\leq i_{0}\leq k-1$ and $1\leq j_{0}\leq k$ such that $M(i_{0},j_{0})$ appears exactly twice in the matrix $M$.
Without loss of generality, we assume that $M(i_{0},j_{0})=1$. We add $(r+1,r+1,\dots,r+1)$ as the $k$-th row of the matrix $M$ and denote the new matrix by $N=(N(i,j))_{i,j}$. To prove that for almost all $t_{1},\dots,t_{r}\in[0,1]$ the column vectors of the following matrix
$$A(t_{1},\dots,t_{r})=(t_{M(i,j)})_{i,j}$$
are affinely independent, it is equivalent to prove that for almost all $t_{1},\dots,t_{r+1}\in[0,1]$, the following matrix
$$B(t_{1},\dots,t_{r+1}):=(t_{N(i,j)})_{i,j}$$
has nonzero determinant, i.e., $\det(B(t_{1},\dots,t_{r+1}))\neq 0$. By a simple calculation, we get that
$$\det(B(t_{1},\dots,t_{r+1}))=f_{2}(t_{2},\dots,t_{r+1})t_{1}^{2}+f_{1}(t_{2},\dots,t_{r+1})t_{1}+f_{0}(t_{2},\dots,t_{r+1}),$$
where $f_{2}$ is, up to sign, the determinant of the minor of $B(t_{1},\dots,t_{r+1})$ that remains after throwing away all columns and rows in which $t_{1}$ appears (actually only two rows and two columns) or $1$ if no row is left.
In the first case, the minor thus formed is a $(k-2)\times(k-2)$ matrix that also satisfies the assumptions of the lemma.
By induction, for almost all $t_{2},\dots,t_{r+1}\in[0,1]$, $f_{2}(t_{2},\dots,t_{r+1})\neq0$. Therefore we conclude that for almost all $t_{1},\dots,t_{r+1}\in[0,1]$, $\det(B(t_{1},\dots,t_{r+1}))\neq 0$. This ends the proof.
\end{proof}

\begin{Lemma}[{\cite[Lemma A.5]{Gut12a}}]\label{continuous extension}
Let $\epsilon>0$, $m\in\mathbb{N}$ and $B$ be a closed subset of a compact metric space $X$.
Let $f':B\to[0,1]^{m}$ and $\tilde{f}:X\to[0,1]^{m}$ be continuous functions such that
$||f'-\tilde{f}_{|_{B}}||_{\infty}<\epsilon$.
Then there exists a continuous function $f:X\to[0,1]^{m}$ such that
$f_{|_{B}}=f'$ and $||f-\tilde{f}||_{ \infty}<\epsilon$.
\end{Lemma}

\begin{proof}
This is an easy application of the Tietze extension theorem
\cite[Theorem 35.1]{munkres2000topology}:
Let $X$ be a compact metric space, $A$ a closed subset of $X$,
and $f:A\to\mathbb{R}$ a continuous function carrying the standard topology,
then there exists a continuous function $F: X\to\mathbb{R}$ such that $F(a)=f(a)$ for any $a\in A$;
moreover, $F$ can be chosen such that $\sup\{|f(a)|:a\in A\}=\sup\{|F(x)|:x\in X\}$.
\end{proof}

\section{An embedding result involving Rokhlin dimension}
In the remaining sections, we fix a positive integer $k$. Let $(X,\mathbb{Z}^{k})$ be a $\mathbb{Z}^{k}$-action, where $X$ is a compact metric space equipped with a metric $d$. For any $\epsilon>0$, the number $\widim_\epsilon(X,d)$ is defined as the smallest number $n\in\mathbb{N}$ such that there exists a finite open cover $\mathcal{U}$ of $X$ whose mesh is at most $\epsilon$, and whose order is $n$. The \textbf{mean dimension} of $(X,\mathbb{Z}^{k})$ is defined as
\begin{equation}\label{mdim}
\mdim(X)=\sup_{\epsilon>0}\lim_{n\to\infty}\frac{\widim_\epsilon(X,d_{[n]})}{n^{k}},
\end{equation}
where $[n]=\{0,\dots,n-1\}^{k}\subset\mathbb{Z}^{k}$ and $d_{[n]}(x,y)=\max_{g\in[n]}d(gx,gy)$ for any positive integer $n$.
The limit on the right side of \eqref{mdim} exists due to the Ornstein--Weiss lemma \cite{ornstein1987entropy}
(see also \cite[Theorem 6.1]{LW} for a detailed proof). It is well known that $\mdim(([0,1]^{m})^{\mathbb{Z}^{k}},\sigma)=m$
\cite[Proposition 3.3]{LW}.

For a continuous function $f:X\to [0,1]^m$, the induced map
$$I_f:X\to ([0,1]^m)^{\mathbb{Z}^{k}}\quad\text{given by}\quad I_f(x)=(f(wx))_{w\in\mathbb{Z}^k}$$
is continuous and equivariant with respect to $(X,\mathbb{Z}^k)$ and $(([0,1]^{m})^{\mathbb{Z}^{k}},\sigma)$.

Let $Y$ be a topological space. For $\epsilon>0$, a continuous map $f:X\to Y$ is called an \textbf{$\epsilon$-embedding} if $d(x,y)<\epsilon$ whenever $f(x)=f(y)$. The following lemma plays a significant role in the proof of Theorem \ref{embedding general}.
\begin{Lemma}[{\cite[Lemma 2.1]{GutTsu12}}] \label{Gutman widim}
Let $(X,d)$ be a compact metric space and $f :X\to [0,1]^m$ a continuous map. Suppose that the numbers $\delta,\epsilon>0$ satisfy the implication
$$d(x,y)<\epsilon\quad\implies\quad\|f(x)-f(y)\|_\infty <\delta.$$
If $\widim_\epsilon(X,d) < m/2$, then there exists an $\epsilon$-embedding $g: X\to [0,1]^m$ satisfying
$$\sup_{x\in X} \|f(x)-g(x)\|_\infty <\delta.$$
\end{Lemma}

\begin{Theorem}[=Theorem 2]\label{embedding general}
Let $D\in\mathbb{N}\cup\{0\}$ and $L\in\mathbb{N}$. Let $(X,\mathbb{Z}^{k},T)$ be an extension of $(Y, \mathbb{Z}^{k},S)$ with the factor map $\pi:X\to Y$.  Assume that $\dim_{Rok}(Y,\mathbb{Z}^{k})=D$ and $\mdim(X)<L/2$.
Then the set of functions
\[
\left\{f\in C(X,[0,1]^{(D+1)L}): I_f\times\pi~\text{is an embedding} \right\}
\]
is a dense $G_\delta$ subset of $C(X,[0,1]^{(D+1)L})$. In particular, there exists an embedding from $(X,\mathbb{Z}^{k},T)$ into $\bigl(([0,1]^{(D+1)L})^{\mathbb{Z}^k}\times Y, \sigma \times S \bigl)$.
\end{Theorem}
\begin{proof}
For each $\eta>0$, define
\[
A_\eta = \{f\in C(X,[0,1]^{(D+1)L}): I_f\times\pi~\text{is an}~\eta\text{-embedding} \}.
\]
Our assertion amounts to showing that the intersection $\bigcap_{\eta>0} A_\eta$ is a dense $G_\delta$ subset of
$C(X,[0,1]^{(D+1)L})$ with respect to the $\|\cdot\|_\infty$-norm. Each $A_\eta$ is obviously open, and moreover this intersection coincides with the countable intersection $\bigcap_{n\in \mathbb{N}} A_{1/n}$. By the Baire category theorem, it thus suffices to show that each set $A_\eta$ is dense.

From now on, let $\eta>0$ and $\delta>0$ be fixed. Let $f\in C(X,[0,1]^{(D+1)L})$. We will show that there is $g\in A_\eta$ with $\|f(x)-g(x)\|_\infty<\delta$ for all $x\in X$.

Under obvious identifications, express the function $f$ as $f=f_0\times\dots\times f_D: X\to [0,1]^{(D+1)L}$ for continuous functions $f_0,\dots,f_D: X\to [0,1]^L$. We choose $\eta\geq\epsilon>0$ satisfying the implication
$$d(x,y)<\epsilon\quad\implies\quad \|f_i(x)-f_i(y)\|_\infty<\delta$$
for any $x,y\in X$ and each $0\leq i\leq D$.
Since $\mdim(X)<L/2$, by \eqref{mdim} we know that there exists $n\in\mathbb{N}$ such that
$$\widim_\epsilon(X,d_{[n]}) < Ln^k/2.$$
For each $0\leq i\leq D$, consider the function $F_i: X\to [0,1]^{Ln^k}=([0,1]^L)^{[n]}$ given by
$$F_i(x)=\bigl( f_i(vx) \bigl)_{v\in [n]}.$$
Then for any $x,y\in X$ and each $0\leq i\leq D$, we have the following
\[
d_{[n]}(x,y)<\epsilon\quad\implies\quad \|F_i(x)-F_i(y)\|_\infty<\delta.
\]
Applying Lemma \ref{Gutman widim}, we can find $\epsilon$-embeddings $G_0,\dots,G_D: X\to [0,1]^{Ln^k}=([0,1]^L)^{[n]}$ with respect to the metric $d_{[n]}$ on $X$ satisfying
\begin{equation}\label{approximation of F}
\|F_i(x)-G_i(x)\|_\infty<\delta\quad\text{for all}~x\in X.
\end{equation}
Using that $\dim_{Rok}(Y, \mathbb{Z}^{k})=D$,
we find $D+1$ open sets $U_{0}, U_{1}, \dots, U_{D}$ such that
$\{vU_{i}\}_{v\in[n]}^{i=0,\dots,D}$ is an open cover of $Y$ as defined in the introduction.
Pulling back this cover via the factor map $\pi$,
we obtain a Rokhlin cover of $X$ via $V^{v}_{i}=\pi^{-1}(vU_{i})$ for all $0\leq i\leq D$ and $v\in[n]$. For each
$0\leq i\leq D$, write $W_{i} = \overline{\bigcup_{v\in[n]}V^{v}_{i}}$.

 For each $v\in[n]$, denote by $p_v: ([0,1]^L)^{[n]}\to [0,1]^L$ the projection onto the $v$-th coordinate.
 For all $0\leq i\leq D$, consider the continuous map
\[
g_i': W_{i} \to [0,1]^L\quad\text{given by}\quad g_i'(x)=p_v\circ G_i((-v)x)\quad\text{for}~x\in\overline{V}^{v}_{i}.
\]
Now fix $0\leq i\leq D$ and $x\in\overline{V}^{v}_{i}$.
By \eqref{approximation of F},
\[
\|f_i(x)-g_i'(x)\|_\infty=\|p_v\circ F_i ((-v)x)-p_v\circ G_i ((-v)x)\|_\infty<\delta.
\]
Thus we have $\|f_i(x)-g_i'(x)\|_\infty<\delta$ on the domain of $g_i'$. Since this domain is compact, Lemma \ref{continuous extension} allows us to find a continuous extension $g_i: X\to [0,1]^L$ of $g_i'$ with $\|f_i(x)-g_i(x)\|_\infty<\delta$ for all $x\in X$.

Consider $g=g_0\times\dots\times g_D: X\to ([0,1]^L)^{(D+1)} = [0,1]^{(D+1)L}$. By construction, we obviously have $\|f(x)-g(x)\|_\infty<\delta$ for all $x\in X$. We claim that $g$ is in $A_\eta$. For this assume that $I_g\times\pi(x)=I_g\times\pi(y)$ for some $x,y\in X$. In particular, $\pi(x)=\pi(y)$. By the definition of the cover $\{V^{v}_{i}\}^{i=0,\dots,D}_{v\in[n]}$, it follows that there is some $0\leq i\leq D$ and $v\in[n]$ with $x,y\in V^{v}_{i}$. Now $I_g(x)=I_g(y)$ implies $I_{g_i}(x)=I_{g_i}(y)$, which by definition implies $g_i(wx)=g_i(wy)$ for all $w\in\mathbb{Z}^k$. For $w\in[n]$, observe that
$(w-v)x \in V^{w}_i$, and thus
\[
g_i((w-v)x) = g_i'((w-v)x)= p_w\circ G_i((-v)x).
\]
The analogous calculation holds for $y$ instead of $x$. Since $w\in[n]$ was arbitrary, it follows that
$G_i((-v)x)=G_i((-v)y)$. By construction, $G_i$ is an $\eta$-embedding with respect to the metric $d_{[n]}$, which implies that
\[
d(x,y) = d(v((-v)x), v((-v)y)) \leq d_{[n]}((-v)x,(-v)y) < \eta.
\]
This finishes the proof.
\end{proof}

\begin{Corollary} \label{embedding findim factor}
Let $(X,  \mathbb{Z}^{k})$ be an extension of a free $\mathbb{Z}^{k}$-action
$(Y,  \mathbb{Z}^{k})$, $D\in\mathbb{N}\cup\{0\}$ and $L\in\mathbb{N}$.
If  $\mdim(X)<L/2$, $\dim_{Rok}(Y, \mathbb{Z}^{k})=D$ and
$Y$ has finite Lebesgue covering dimension,
then there exists an embedding from $(X, \mathbb{Z}^{k})$ into
$\bigl( ([0,1]^{(D+1)L+1})^{\mathbb{Z}^k}, \sigma\bigl)$.
\end{Corollary}
\begin{proof}
Since in this case, $(Y,\mathbb{Z}^{k})$ embeds into $([0,1]^{\mathbb{Z}^k},\sigma)$ by \cite[Theorem 4.2]{J74}, this follows directly from Theorem \ref{embedding general}.
\end{proof}

The third named author \cite[Corollary 5.2]{Sza15}
proved that for a free $\mathbb{Z}^{k}$-action $(X, \mathbb{Z}^{k})$,
$\dim_{Rok}(X,\mathbb{Z}^{k})\leq2^k(\dim(X)+1)-1$.
As a consequence we have
\begin{Corollary}
Let $(X, \mathbb{Z}^{k})$ be an extension of a free $\mathbb{Z}^{k}$-action $(Y,  \mathbb{Z}^{k})$
and $L\in\mathbb{N}$.
If $Y$ has finite Lebesgue covering dimension and $\mdim(X)<L/2$,
then there exists an  embedding from $(X, \mathbb{Z}^{k})$ into
$\bigl( ([0,1]^M)^{\mathbb{Z}^k}, \sigma\bigl)$,
where $M=1+2^k(\dim(X)+1)L$.
\end{Corollary}

Let $\alpha_{1},\dots,\alpha_{k}$ be irrational numbers.
An \textbf{irrational $\mathbb{Z}^k$-rotation} on the $k$-torus $\mathbb{T}^{k}$ is defined as
$$T_{\alpha_{1},\dots,\alpha_{k}}:\mathbb{Z}^k\times\mathbb{T}^{k}\to\mathbb{T}^{k},
(n_{1},\dots, n_{k})\times(x_{1},\dots, x_{k})\mapsto(x_{1}+n_{1}\alpha_{1},\dots, x_{k}+n_{k}\alpha_{k}).$$
Here $\alpha_{1},\dots,\alpha_{k}$ do not have to be linearly independent over the rationals.

\begin{Corollary}
Let $(X,\mathbb{Z}^{k})$ be an extension of an irrational $\mathbb{Z}^k$-rotation
on the $k$-torus and $L\in\mathbb{N}$.
If $\mdim(X)<L/2$, then there exists an embedding from
$(X, \mathbb{Z}^{k})$ into $\bigl( ([0,1]^{2^kL+1})^{\mathbb{Z}^k}, \sigma\bigl)$.
\end{Corollary}
\begin{proof}
It can be proved that any irrational rotation on $\mathbb{T}$ is of topological Rokhlin dimension $1$
as a $\mathbb{Z}$-action
(\cite[Theorem 6.2]{HWZ12}).
By the definition of Rokhlin dimension, we know that for every irrational rotation
$(\mathbb{T}, T_{\alpha_{i}})$ $(1\leq i\leq k)$ and $n\in\mathbb{N}$,
there exist open sets $U_{0}^{(i)}$ and $U_{1}^{(i)}$ in $\mathbb{T}$ such that
the following hold:
\begin{enumerate}
\item for every $m\in\{0,1\}$, $T_{\alpha_{i}}^{j}\overline{U}_{m}^{(i)}$ $(0\leq j\leq n-1)$ are pairwise disjoint;

\item $\bigcup_{m=0}^{1}\bigcup_{j=0}^{n-1}T_{\alpha_{i}}^{j}U_{m}^{(i)}=\mathbb{T}$.

\end{enumerate}
Consider the following $2^{k}$ open sets in $\mathbb{T}^{k}$:
$$U_{m_{1}}^{(1)}\times U_{m_{2}}^{(2)}\times\cdots\times U_{m_{k}}^{(k)}, \;\;\; m_{i}\in \{0,1\}.$$
Using (1) and (2), one can easily check that for every $n\in\mathbb{N}$, every open set above induce an $[n]$-tower and the union of these towers covers $\mathbb{T}^{k}$. So the Rokhlin dimension of the irrational rotation induced by $\alpha_{i}$ $(1\leq i\leq k)$ is at most $2^{k}-1$. Now apply Theorem \ref{embedding findim factor} to finish the proof.
\end{proof}


\section{Takens' embedding theorem with a continuous observable
and an embedding conjecture for $\mathbb{Z}^{k}$-actions}
The proof of the following result, Takens' embedding theorem with a continuous observable for $\mathbb{Z}^{k}$-actions,
is analogous to that of Theorem \ref{takens embedding for Z}. We thus do not give any further details.
\begin{Theorem}\label{takens embedding for Z^k}
Let $d\in\mathbb{N}\cup\{0\}$ and $m\in\mathbb{N}$.
If a dynamical system $(X,\mathbb{Z}^{k})$ satisfies that $\dim(X)=d$
and
$$\frac{\dim(\{x\in X:ax=x,a\in A\})}{[\mathbb{Z}^{k}:A]}<\frac{m}{2}$$
for every subgroup $A$ of $\mathbb{Z}^{k}$ with
$[\mathbb{Z}^{k}: A]\leq 2d$,
then the set of continuous functions $f:X\to [0,1]^{m}$ so that
\begin{equation*}\label{deff2d}
f_{2d}:X\to ([0,1]^{m})^{[0,2d]^{k}\cap\mathbb{Z}^{k}},\;\;x\mapsto (f(ix))_{i\in[0,2d]^{k}\cap\mathbb{Z}^{k}}
\end{equation*}
is an embedding
is comeagre in $C(X,[0,1]^{m})$ w.r.t. supremum topology.
\end{Theorem}
Heuristically, Theorem \ref{takens embedding for Z^k} (for $k=2$) corresponds to an experimental setup, where the system may be subjected to an external change $S$ (e.g., a magnetic field) which commutes with the time evolution map $T$, i.e., $TS=ST$. In \cite{GutQiaTsu2017} a general embedding conjecture for $\mathbb{Z}^{k}$-actions is presented.
For a $\mathbb{Z}^{k}$-action $(X,\mathbb{Z}^{k})$ and a subgroup $A$ of $\mathbb{Z}^{k}$, we define
$$X_{A}=\{x\in X: nx=x, n\in A\}$$ and let $\mathbb{Z}^{k}/A$ be the the quotient group of $\mathbb{Z}^{k}$ by $A$.
Then $\mathbb{Z}^{k}/A$ acts on the space $X_{A}$ in the following natural way: $(aA)x=ax$, $aA\in\mathbb{Z}^{k}/A$ and $x\in X_{A}$. So $(X_{A},\mathbb{Z}^{k}/A)$ is also a dynamical system and the mean dimension of $(X_{A},\mathbb{Z}^{k}/A)$ is well defined. Moreover, if $\mathbb{Z}^{k}/A$ is a finite group, that is, the index of $A$ in $\mathbb{Z}^{k}$ is finite,
then by the definition of mean dimension we easily get that
\begin{equation}\label{perdim}
\mdim(X_{A},\mathbb{Z}^{k}/A)=\frac{\dim(X_{A})}{\#(\mathbb{Z}^{k}/A)},
\end{equation}
where $\dim(\cdot)$ is the Lebesgue covering dimension introduced in Section 2.
The conjecture on embedding for $\mathbb{Z}^{k}$-actions in \cite{GutQiaTsu2017} is as follows:
\begin{Conjecture}\label{conjectureZ^k}
If a dynamical system $(X,\mathbb{Z}^{k})$ satisfies that for every subgroup $A$ of $\mathbb{Z}^{k}$
$$\mdim(X_{A},\mathbb{Z}^{k}/A)<\frac{D}{2},$$
then $(X,\mathbb{Z}^{k})$ embeds into the $D$-cubical shift $(([0,1]^{D})^{\mathbb{Z}^{k}}, \sigma)$.
\end{Conjecture}
Note that every subgroup of $\mathbb{Z}$ has the form $n\mathbb{Z}$ ($n\in\mathbb{Z})$
and that the index of $n\mathbb{Z}$ in $\mathbb{Z}$ is $|n|$.
For $k=1$, Conjecture \ref{conjectureZ^k} coincides with Conjecture \ref{conjectureLinTsu}.
Observe that the mean dimension of finite dimensional spaces is zero.
As an immediate consequence of Takens' embedding theorem for $\mathbb{Z}^{k}$-actions (Theorem \ref{takens embedding for Z^k}), we confirm the correctness of Conjecture \ref{conjectureZ^k} for finite dimensional dynamical systems. Moreover, we state without proof the following related result:
\begin{Theorem}\label{period immersion}
Let $(X, \mathbb{Z}^{k})$ be a $\mathbb{Z}^{k}$-action and $m\in\mathbb{N}$.
If for every subgroup $H$ of $\mathbb{Z}^{k}$ with $[\mathbb{Z}^{k}:H]<\infty$,
$$\frac{\dim(\{x\in X: hx=x,h\in H\})}{[\mathbb{Z}^{k}:H]}<\frac{m}{2},$$
then there is an equivariant immersion
\footnote
{By an immersion we mean an injective continuous mapping.
Note that when the domain is not compact, this may not be an embedding.}
from $\bigcup_{[\mathbb{Z}^{k}:H]<\infty} X_{H}$ to $([0,1]^m)^{\mathbb{Z}^{k}}$.
\end{Theorem}

\appendix
\section{The Lindenstrauss--Tsukamoto Conjecture\\ holds generically}
Let $Q=[0,1]^{\mathbb{N}}$ be the Hilbert cube and consider the infinite product
$$Q^{\mathbb{Z}}=\cdots\times Q\times Q\times Q\times\cdots.$$
The space $Q^{\mathbb{Z}}$ is metrizable in the product topology and by Tychonoff's theorem is compact.
We denote by $d$ a metric on $Q^{\mathbb{Z}}$ inducing the product topology.
Define the shift $\sigma$ on $Q^{\mathbb{Z}}$ by
$$\sigma((x_{n})_{n\in\mathbb{Z}})=(x_{n+1})_{n\in\mathbb{Z}}, \;\;\; \text{ where } x_{n}\in Q.$$

Every compact metric space is homeomorphic to a subspace
of the Hilbert cube $Q$ \cite[Theorem 4.14]{kechris2012classical}.
For a dynamical system $(X,T)$, let $f: X\to Q$ be a topological embedding of $X$ into $Q$.
We define $I_{f}: X\to Q^{\mathbb{Z}}$ by
$$I_{f}(x)=(f(T^{n}x))_{n\in\mathbb{Z}}.$$
One can readily check that $(X, T)$ embeds equivariantly into $(Q^{\mathbb{Z}},\sigma)$ via $I_{f}$.
So $(I_{f}(X), \sigma)$ is isomorphic to $(X, T)$ via $I_{f}$ and
hence we may regard $(X, T)$ as a subsystem of $(Q^{\mathbb{Z}},\sigma)$.

Let
$$S=\{X\subset Q^{\mathbb{Z}}: X \text{ is closed, non-empty and }\sigma\text{-invariant}\}$$
be the space of all subsystems of $(Q^{\mathbb{Z}},\sigma)$.
This space is compact in the Hausdorff metric, which we will denote by $D_{H}$.
We associate each $X\in S$ with the dynamical system $(X,\sigma|X)$,
making $S$ into a parametrization of dynamical systems.

Let $K$ be the Cantor set and
$\text{Homeo}(K)$ the collection of all homeomorphisms from $K$ to itself.
Kechris and Rosendal \cite{kechris2007turbulence} found $\psi\in\text{Homeo}(K)$
such that its isomorphism class
$\{\phi\in\text{Homeo}(K):(K,\phi)\cong(K,\psi)\}$
is comeagre in $\text{Homeo}(K)$.
We call such a system the Kechris--Rosendal system.
Soon afterwards, Akin, Glasner and Weiss described this system explicitly in \cite{akin2008generically}.
Hochman proved
\begin{Theorem}[{\cite[Corollary 3.6]{hochman2008genericity}}]\label{genericity}
The Kechris--Rosendal system is generic in $S$.
\end{Theorem}
We now claim that the Kechris--Rosendal system, denoted by $(X, \sigma)$ in $S$, is aperiodic.
Fix $Y\in S$, where $(Y, \sigma)$ is aperiodic.
Suppose that $(X,\sigma)$ has a periodic point $x$ of period $k$.
By the density of the Kechris--Rosendal systems in $S$,
we can find $(X_{n}, \sigma)$ $(n\geq 1)$ which are isomorphic to $(X,\sigma)$
such that
$X_{n}\to Y$ w.r.t. $D_{H}$.
Let $x_{n}\in X_{n}$ be of period $k$ and $y$ an accumulation point of the sequence $\{x_{n}\}_{n}$.
Clearly, $y$ is also a periodic point of period less than or equal to $k$.
Without loss of generality, assume $x_{n}\to y$.
By a simple calculation, we get that
\begin{align*}
\begin{split}
d(y, Y)\leq d(y, X_{n})+D_{H}(X_{n}, Y)\leq & d(y,x_{n})+D_{H}(X_{n}, Y)\to 0, \text{ as } n\to \infty.
\end{split}
\end{align*}
So $y\in Y$, a contradiction.
Therefore the Kechris--Rosendal system is aperiodic.
It is well known that the Cantor set is zero dimensional.
By the classic theorem due to Jaworski \cite{J74}, we know that the Kechris--Rosendal system
can be embedded into the $1$-cubical shift (i.e., Conjecture \ref{conjectureLinTsu} holds for the Kechris--Rosendal system).
Thus, by Theorem \ref{genericity} we get that Conjecture \ref{conjectureLinTsu} holds generically.

\begin{Remark}
We say that a dynamical system $(X,T)$ has the \textbf{marker property} if for every natural number $N$ there exists an open set $U\subset X$ satisfying that $U\cap T^{-n}U=\emptyset$ ($0<|n|<N$) and $X=\cup_{n\in\mathbb{N}}T^{n}U$. This property obviously implies the aperiodicity of $(X,T)$. Since the Kechris--Rosendal system is aperiodic zero-dimensional,
it has the marker property by \cite[Lemma 3.3]{GutTsu12} (or \cite[Theorem 6.1]{Gut12a} for a stronger result).
Thus by Theorem \ref{genericity} the marker property holds generically. 
\end{Remark}
\bibliographystyle{alpha}
\bibliography{universal_bib}

\end{document}